\documentclass[a4paper,11pt]{amsart}

\usepackage{graphicx,import}
\usepackage[english]{babel}

\usepackage{url}
\usepackage[pdftitle={An optimal transport approach for solving dynamic inverse problems in spaces of measures},pdfauthor={Kristian Bredies and Silvio Fanzon}]{hyperref}

\usepackage{mathrsfs}

\usepackage{todonotes}
\usepackage[T1]{fontenc}
\usepackage[utf8]{inputenc}
\usepackage{csquotes}
\usepackage{verbatim}
\usepackage{todonotes}
\usepackage{amsmath}
\usepackage{amsthm}
\usepackage{amsfonts, amssymb, fancybox, multicol, array, tabularx, amscd}
\usepackage{enumerate}
\usepackage{upgreek}
\usepackage[shortlabels]{enumitem}
\usepackage{tikz}
\usepackage{enumitem}
\newtheorem{theorem}{Theorem}[section]
\newtheorem{lemma}[theorem]{Lemma}
\newtheorem{proposition}[theorem]{Proposition}
\newtheorem{corollary}[theorem]{Corollary}

\theoremstyle{definition}
\newtheorem{definition}[theorem]{Definition}
\newtheorem{assumption}[theorem]{Assumption}
\newtheorem{remark}[theorem]{Remark}
\newtheorem{example}[theorem]{Example}

\theoremstyle{remark}

\newcommand{\F}{\mathscr{F}}
\newcommand{\R}{\mathbb{R}}
\newcommand{\N}{\mathbb{N}}

\newcommand{\C}{\mathbb{C}}

\newcommand{\e}{\varepsilon}
\newcommand{\f}{\varphi}
\newcommand{\Ltwo}{L^2([0,1];H)}
\newcommand{\Lp}{L^p([0,1];H)}

\DeclareRobustCommand{\rchi}{{\mathpalette\irchi\relax}}
\newcommand{\irchi}[2]{\raisebox{\depth}{$#1\chi$}}

\newcommand{\zak}{%
  \mathbin{\vrule height 1.6ex depth 0pt width
0.13ex\vrule height 0.13ex depth 0pt width 1.3ex}
}    %

\newcommand{\nor}[1]{\left\| #1 \right\|} %
\newcommand{\norh}[1]{{\left\| #1 \right\|}_{H_t}}

\newcommand\inn[2]{{\langle #1, #2 \rangle}}

\newcommand\inner[2]{{\langle #1, #2 \rangle}_{H_t}}
\newcommand\innerL[2]{{\langle #1, #2 \rangle}_{L^2}}

\DeclareMathOperator*{\argmin}{arg\,min}

\DeclareMathOperator{\Div}{div}

\DeclareMathOperator{\supp}{supp}

\DeclareMathOperator*{\esssup}{ess\,sup}

\newcommand{\curves}{C_{\rm w} ([0,1];\M(\overline{\Om}))}
\newcommand{\pcurves}{C_{\rm w} ([0,1];\M^+(\overline{\Om}))}
\newcommand{\B}{\mathcal{B}}
\newcommand{\M}{\mathcal{M}}

\newcommand{\weak}{\rightharpoonup}
\newcommand{\weakstar}{\stackrel{*}{\rightharpoonup}}
\newcommand{\Om}{\Omega}
\newcommand{\olom}{\overline{\Omega}}

\newcommand{\de}{\partial}

 \title[An optimal transport approach for dynamic inverse problems]{An optimal transport approach for solving dynamic inverse problems in spaces of measures} 
 \author[K. Bredies]
 {Kristian Bredies}
 \address[Kristian Bredies]{University of Graz, Institute of Mathematics and Scientific Computing, Heinrichstra\ss e 36, 8010 Graz, Austria}
\email{Kristian.Bredies@uni-graz.at}

\author[S. Fanzon]
{Silvio Fanzon}
\address[Silvio Fanzon]{University of Graz, Institute of Mathematics and Scientific Computing, Heinrichstra\ss e 36, 8010 Graz, Austria}
 \email{Silvio.Fanzon@uni-graz.at}

\begin{document}

\begin{abstract}
\small{
In this paper we propose and study a novel optimal transport based regularization of linear dynamic inverse problems. The considered inverse problems aim at recovering a measure valued curve and are dynamic in the sense that (i) the measured data takes values in a time dependent family of Hilbert spaces, and (ii) the forward operators are time dependent and map, for each time, Radon measures into the corresponding data space. 
The variational regularization we propose is based on dynamic (un-)balanced optimal transport which means that the measure valued curves to recover (i) satisfy the continuity equation, i.e., the Radon measure at time $t$ is advected by a velocity field $v$ and varies with a growth rate $g$, and (ii) are penalized with the kinetic energy induced by $v$ and a growth energy induced by $g$.
We establish a functional-analytic framework for these regularized inverse problems, prove that minimizers exist and are unique in some cases, and study regularization properties.
This framework is applied to dynamic image reconstruction in undersampled magnetic resonance imaging (MRI), modelling relevant examples of time varying acquisition strategies, as well as patient motion and presence of contrast agents.

}
\end{abstract}
 
\maketitle
\tableofcontents

\section{Introduction}
In this paper we are concerned with solving ill-posed dynamic inverse problems where the sought unknown is a curve of Radon measures. We propose to regularize such problems via balanced and unbalanced dynamic optimal transport and establish a functional-analytic framework that takes the specificities of dynamic inverse problems, such as the time-varying nature of the measurement process, into account. Well-posedness as well as regularization properties are proven, and the application
to magnetic resonance imaging (MRI) is discussed.

Our motivation to consider dedicated strategies for dynamic inverse problems
arises from the shortcomings of static %
reconstruction strategies for
inverse problems in the presence of motion during the measurement process.
In the %
static case, measurement data is usually continuously collected such that
 sufficiently many data is available to enable the unique solution of
 the underlying inverse problem. In this context, one has to assume that no dynamics occur during the measurement process.
 However, this assumption is often violated for many applications, including
 medical imaging techniques such as MRI %
 and
 computed tomography (CT) that image, e.g., the beating heart or the lung while breathing.
 Consequently, a consistent reconstruction is
 no longer possible and static approaches usually admit motion artifacts.
 A strategy to overcome this is to temporally resolve the dynamics, 
 meaning that for each time
 instance during the measurement, one seeks to reconstruct a solution where
 only a small fraction of the
 necessary data is available. In addition to that, generally, in each time instance,
 a different part of the data set is  acquired. This results in a dynamic inverse problem with time-variant forward operators and data spaces, where for each fixed time instance, the corresponding
 inverse problem is massively underdetermined. In order to solve such a challenging problem, both an appropriate dynamic regularization strategy
 as well as a suitable modelling of forward operators and data spaces is necessary.

 We propose and study a regularization strategy that bases on optimal transport energies, both in a balanced and unbalanced context, see below for a detailed
 description. Such strategies are naturally linked with 
 curves of Radon measures, inverse problems in the space of Radon measures
 and appropriate Radon-norm-based regularizers.
 Indeed, the fact that for each point in time, an inverse problem
 with underdetermined data has to be solved calls for dedicated regularization
 such as the intensively studied sparsity-promoting $\ell_1$-type penalties. 
 In the discrete setting, this leads to the celebrated theory of compressed-sensing \cite{crt,donoho2006compressedsensing}, in which one is able to reconstruct the unknown starting from very few random measurements, yielding better stability properties. The continuous, infinite-dimensional counterpart %
 is given by the space of Radon measures \cite{bp,fernandezgranda2014superres}, where the regularization can be achieved penalizing the Radon norm
 and formulating the inverse problem in measure space.
 Recovering the unknown from very few observations is then possible since the data admits redundancies, particularly in applications to medical imaging. 

 In the dynamic setting,
 data redundancy additionally needs to be exploited by taking time correlation  into account. Indeed, one can expect displacements between %
 consecutive time samples to be small, and incorporate this information in the regularizer in order to achieve better reconstruction. In particular, the fluid-mechanics formulation of both balanced and unbalanced optimal transport, known as Wasserstein \cite{ags,bb} and Wasserstein--Fisher--Rao distance \cite{chizat,kmv,liero}, respectively, are particularly well-suited to keep track of motion and possible mass change occurring in the ground truth. Further, this formulation is based on curves of Radon measures and is thus attractive for the regularization of dynamic inverse problems, where in each time instance, a Radon measure should be recovered.
 The regularization is then enforced by subjecting potential solutions of
 the dynamic problem, i.e., curves of Radon measures, to
 the continuity equation, possibly with source, while at the same time penalizing displacement field and growth rate. As we will see, this
 approach indeed establishes a convex regularization strategy that
 is
 sparsity-promoting in each time frame, exploits data redundancy in time and
 intrinsically recovers the velocity field associated with the motion as well
 as the rate of brightness changes.
 Let us emphasize that in particular, the approach allows for continuous measurement in time while providing spatial Radon-norm regularization. In contrast, a straightforward generalization of \cite{bp} to the space-time cylinder would, e.g., allow for measures that are singular in time and hence, not regular enough to consistently define global-in-time forward operators and data discrepancies.
 
 Another aspect that has to be considered for dynamic inverse
 problems is a faithful modelling of the measurement process with respect
 to time,
 taking into account the time-varying nature of the measurements.
 In this paper, the latter is achieved by the construction of ad hoc Bochner-type spaces in which the data can take values in a time-dependent family of Hilbert spaces, which are correlated in time in a very weak way. This enables us to
 model the dynamic inverse problem with a time-dependent family of linear forward operators, mapping Radon measures to the associated data Hilbert space
 in each time instance, by a global-in-time forward operator
 that takes curves of Radon measures to the time-varying measurement data,
 making it thus possible to consistently define a data-mismatch term.
 In this respect, our model is truly dynamic
 and well-adapted to undersampled data as outlined above.

 The overall approach then %
 realizes a reconstruction
 by inverting the global %
 forward operator subject to optimal transport penalization and continuity equation. Such an approach leads to a well-posed and convex variational problem of Tikhonov type, in which we are able to reconstruct the sought solution, along with the displacement field and mass growth rate.
 The main task of the paper is to establish a rigorous functional-analytic framework in which to set the problem and to obtain well-posedness and stability properties for the proposed variational optimal transport-based regularization. We then apply our theoretical results to dynamic medical imaging, focusing on the case of undersampled MRI, %
 showing that we are able to treat an almost arbitrary variety of sampling strategies, as well as being principally able to reconstruct the image sequence, recover the motion displacement and track the possible %
 presence of contrast agents.
 
 Let us shortly review the existing literature on dynamic inverse problems
 and optimal transport approaches for inverse problems, image processing as well as computer vision.
 While the theory of regularization of dynamic inverse problems is a relatively new field of research \cite{shb}, regularization theory for static ill-posed inverse problems dates back several decades.
 In this context, the approach in this paper can be classified as
 Tikhonov regularization in Banach space, a well-established technique
 where one penalizes the data mismatch by a convex regularizer and solves the corresponding variational problem \cite{ehn,skhk,ta,tly}.
For computer vision and
image processing applications, research in the last decades focused on
specific convex regularizers \cite{sgghf} such as, for instance, edge-preserving
functionals (total variation \cite{cl,rof}, total generalized variation
\cite{brediesholler,bkp}), or functionals that enforce sparsity with 
respect to a given
basis, frame or learned dictionary \cite{aeb,crt,ddd,frikel,vgvvpvs}.
In this context, Radon-norm penalties can also be interpreted
as sparsity-promoting regularization \cite{bp,fernandezgranda2014superres}.

Concerning optimal transport, the classical theory deals with the problem of transporting mass from a probability distribution into a target one, while minimizing, e.g., the average squared displacement. Such a minimization problem defines a metric over the space of probability measures, called Wasserstein distance \cite{ags,cuturi2019,sant,villani}. In \cite{bb}, the authors showed that the Wasserstein distance can be computed via solving a convex dynamic problem that corresponds to finding a
geodesic path in the space of probabilities subjected to the continuity equation that minimizes the kinetic energy. This formulation is the basis for the balanced optimal-transport regularization studied in this paper.
Such approach, however, intrinsically assumes mass constancy, which is not always desired in applications, e.g., in mathematical imaging.
In the recent years, several ways to overcome this limitation were proposed, leading to
so-called unbalanced optimal transport \cite{benamou03,figalli10,figalli10bis,GANGBO2019108940,lm2015,mrss15,pr14,pr16}. %
In this context, common strategies are to
add a source term to the continuity equation and consequently, to the kinetic energy, or to allow mass to escape/enter the domain, by interpreting the boundary as an infinite reservoir.
In this paper, the energy introduced independently in \cite{chizat,kmv,liero}, known as Wasserstein--Fisher--Rao or Hellinger--Kantorovich, is used to provide a unbalanced optimal-transport regularization.  The remarkable feature of such formulation is that geodesics have a clear meaning, as they can be interpreted as joint displacement and change of mass, and therefore capture the dynamics of, e.g., image sequences.

Returning to inverse problems, as already mentioned, research in the dynamic framework  
recently gained some momentum \cite{shb},
where convex regularizers that penalize the time derivative, interpret the space-time cylinder as a higher-dimensional set or enforce a spatio-temporal decomposition of low rank have been studied in the literature, most prominently in the context of medical imaging applications \cite{dlwk,hk,lhdj,ocs,shsbs,sl,ws}.
In comparison, such approaches, however, do only implicitly account for motion information in contrast to the proposed optimal-transport regularizer which explicitly yields a motion field.
In this respect, the
employment of optimal transport energies as regularizers for inverse
problems is a very recent development. Here, existing literature
mainly focuses on static inverse problems and static optimal transport
leading, for instance, to Wasserstein-distance type regularization
\cite{kr,mbmoj}. In contrast, dynamic optimal transport has been utilized for
specific image processing and computer vision tasks such as image
interpolation \cite{chizat,hmp,ppo}. We also mention the work \cite{schmitzer}, which appeared after the present work. In \cite{schmitzer}, the authors propose to regularize an inverse problem related to PET image reconstruction through balanced optimal transport, subsequently applying it to the problem of tracking radiolabelled cells. The regularizer they propose is similar to ours, however the forward operator they consider is static and application-specific, whereas we are able to deal with general dynamic inverse problems. Moreover, their %
analytical framework is greatly simplified, dealing only with discretized unknowns in space-time satisfying a discrete version of the continuity equation, rather than with actual curves of measures, which is the natural framework to obtain well-posedness, as proposed in the present paper. %
To the best
knowledge of the authors, no other works %
employ dynamic optimal transport regularization for dynamic inverse problems.
In particular, a framework for recovering curves of Radon measures from
continuously acquired measurements does not exist to date.
Let us also mention that the realization of the time-dependent Bochner
spaces introduced in this paper is new. Indeed, existing approaches
usually assume that almost every data space %
is isomorphic, which is
sufficient to model, i.e., function spaces over time-varying domains
\cite{elliott, toader}. Such isomorphy is not required in our approach
which can
thus be used to model very general data acquisition strategies.

The paper is organized as follows.
In the remainder of this section, we precise the mathematical setting employed for regularizing dynamic inverse problems and summarize the main theoretical results obtained (Section \ref{sec:intro:theory}), including details on the MRI application (Section \ref{sec:intro:mri}). %
In Section \ref{sec:optimal} we lay the theoretical foundations to rigorously define the optimal transport regularizer. %
In Section \ref{sec:hilbert setting} we introduce and study the above mentioned class of time dependent Bochner spaces, which will be used to model the data measurements. After this preliminary part, in Section \ref{sec:regularization}, we  introduce the Tikhonov regularization for the dynamic inverse problem and show well-posedness as well as regularization properties. 
In Section \ref{sec:MRI} we apply our theoretical results to dynamic MRI, also providing examples of sampling strategies.  Finally, Section~\ref{sec:conclusions} concludes with some perspectives for future research and some comments on the related paper \cite{extremal} as well as forthcoming work, in which we perform numerical analysis for the model proposed in this paper.

\subsection{Outline of the mathematical setting and main theoretical results} \label{sec:intro:theory}

Let $\Om \subset \R^d$ be an open and bounded domain, with $d \in \N$, $d \geq 1$, and consider a time variable $t \in [0,1]$. Let $H_t$ be a time-dependent collection of Hilbert spaces modelling the data. The time regularity required for such family will be very mild, as specified below. %
At each time instance $t$ corresponds a given linear continuous forward operator $K_t^*$, mapping from the space of Radon measures $\M(\olom)$ into %
$H_t$. We consider the following inverse problem: Given some data $f_t \in H_t$ for $t \in [0,1]$, find a curve of Radon measures $t \in [0,1] \mapsto \rho_t \in \M(\olom)$ such that
\begin{equation} \label{intro inverse}
K_t^* \rho_t = f_t \, \quad \text{ for a.e. } \quad t \in [0,1] \,.
\end{equation}
We propose to regularize \eqref{intro inverse} by means of balanced/unbalanced  optimal transport. This is enforced by subjecting $\rho_t$ to the continuity equation
\begin{equation}
  \label{eq:intro_cont_eq}
  \de_t \rho_t + \Div (v_t \rho_t ) = g_t \rho_t  \quad \text{ in } \quad (0,1) \times \olom \,,
\end{equation}
where $v_t (x) \colon (0,1) \times \olom \to \R^d$ is a flow field transporting the mass $\rho_t$, while $g_t(x) \colon (0,1) \times \olom \to \R$ is a growth rate keeping track of mass creation and destruction, thus allowing for local mass change. %
We point out that no initial conditions are prescribed on $\rho_t$ in \eqref{eq:intro_cont_eq}, since in the context of the inverse problem~\eqref{intro inverse} we only have available indirect measurements on the whole time interval $[0,1]$. 
We propose to regularize \eqref{intro inverse} by minimizing the Tikhonov functional 
\begin{equation}
\min_{\rho_t, v_t, g_t} \ \frac12 \int_0^1 \|K_t^* \rho_t - f_t
\|_{H_t}^2 \,dt 
+\frac{\alpha}{2} \int_0^1 \int_{\olom}  |v_t(x)|^2 +  \delta^2 
 |g_t(x)|^2  \, d\rho_t(x) \, dt + \beta \int_0^1 \rho_t(\olom) \, dt   \,,\label{eq:intro_tikh} 
\end{equation}
subject to \eqref{eq:intro_cont_eq}.
Here, $\alpha,\beta > 0$ are regularization parameters, $\delta \in (0, \infty]$ is a penalty parameter %
 and the optimization is done for the triple $(\rho_t,v_t,g_t)$.  
The second term in \eqref{eq:intro_tikh} is known in the literature as Wasserstein--Fisher--Rao energy for unbalanced optimal transport \cite{chizat,kmv,liero}, and as Benamou--Brenier energy \cite{bb} for balanced optimal transport when %
$\delta=\infty$, enforcing $g_t=0$ and hence mass preservation.

Our main task is to establish
problem~\eqref{eq:intro_tikh} subject to~\eqref{eq:intro_cont_eq} as a regularizer for \eqref{intro inverse} in a rigorous functional-analytic framework. 
In the following we provide some details on how to make the terms appearing in \eqref{eq:intro_tikh} rigorous, in particular providing suitable assumptions on $K_t^*$ and $H_t$. The natural space in which to cast \eqref{eq:intro_tikh} is given by $\M := \M(X) \times \M(X;\R^d) \times \M(X)$ where $X:=(0,1) \times \olom$. For $(\rho,m,\mu) \in \mathcal{M}$ define the transport energy as the 1-homogeneous convex functional 
\begin{equation} \label{intro convex}
B_\delta (\rho,m,\mu) := \int_{X} \, \Psi_\delta \left( \frac{d\rho}{d\lambda}, \frac{dm}{d\lambda}, \frac{d\mu}{d\lambda} \right) \, d\lambda \,,
\end{equation}
where $\lambda \in \M^+(X)$ is any positive measure such that $\rho, m,\mu \ll \lambda$ and for $(t,x,y) \in \R \times \R^d \times \R$ we define $\Psi_\delta(t,x,y):=\frac{|x|^2+\delta^2 y^2}{2t}$ if $t>0$, $\Psi_\delta(0,0,0):=0$ and $\Psi_\delta(t,x,y):=\infty$ in all other cases. %
Introduce the affine set $\mathcal{D}:=\left\{ (\rho,m,\mu) \in \M \, \colon \, \de_t \rho + \Div m = \mu  \right\}$ where the continuity equation is in the distributional sense, without initial conditions (Definition \ref{def:meas sol}). Whenever $(\rho,m,\mu) \in \mathcal{D}$ and $B_\delta(\rho,m,\mu)<\infty$, it follows that $\rho \geq 0$, $m=v \rho$ and $\mu=g \rho$. Moreover $\rho=dt \otimes \rho_t$ with $t \mapsto \rho_t \in \M(\olom)$ narrowly continuous, i.e.,
$t \mapsto \int_{\olom} \f (x) \, d \rho_t(x)$ is continuous for all $\f \in C(\olom)$ (Proposition \ref{prop:cont rep}). By setting $\lambda = \rho$ in \eqref{intro convex} we recover the second term in \eqref{eq:intro_tikh} (Proposition \ref{prop:bb prop}). %
Next, we outline how we define the space of measurements. Assume given a family of real Hilbert spaces $\{H_t\}_t$ for $t \in [0,1]$, with inner products denoted by $\inner{\cdot}{\cdot}$, satisfying the following. %
\begin{assumption} \label{intro:assumption:1} There exist a Banach space $D$ and linear continuous operators $i_t \colon D \to H_t$ with the properties:
	\begin{enumerate}[label=\textnormal{(H\arabic*)}]
	\item $\nor{i_t} \leq C$ for some constant $C>0$ not depending on $t$,\label{H1i}
	\item $i_t(D)$ is dense in $H_t$,%
	\item the map $t \in [0,1] \mapsto \inner{i_t \f}{ i_t \psi} \in \R$ is Lebesgue measurable for every fixed $\f,\psi \in D$.\label{H3i}
        \end{enumerate}
	\end{assumption}
In other words, we assume that each $H_t$ possesses a dense subset $i_t(D)$, and such subsets are related by the time-measurability condition \ref{H3}. %
In particular, Assumption \ref{intro:assumption:1} allows us to define suitable notions of strong measurability and integrability for measurements $f \colon [0,1] \to H$ for $H := \cup_{t \in [0,1]} H_t$ such that $f_t \in H_t$ for a.e.~$t$ in $[0,1]$ (see Definitions \ref{def:strong}, \ref{def:int}), leading to the definition of the measurements space 
\begin{equation} \label{intro:L2}
\Ltwo := \left\{ f \colon [0,1] \to H \, \colon \, f \ \text{strongly measurable}, \int_0^1 \norh{f_t}^2 \,dt < \infty \right\}\,.
\end{equation}
	In Theorem \ref{thm:completeness} we show that \eqref{intro:L2} is a Hilbert space with $\inn{f}{g}_{L^2}:=  \int_0^1 \inn{f_t}{g_t}_{H_t} \, dt$.  
Notice that our construction provides a natural extension of the classic Bochner theory to the case of varying codomains, in the sense that \eqref{intro:L2} coincides with the classical Bochner space when $H_t = H$ for all $t$, with $H$ given Hilbert space.  Details about the above construction are contained in Section \ref{sec:hilbert setting}.
We now address the assumptions we make on the forward operators $K_t^*$ appearing in \eqref{intro inverse}. %
	
	\begin{assumption}
 For a.e.~$t \in [0,1]$ the linear continuous operators $K_t^* \colon \M(\olom) \to H_t$ satisfy:
	\begin{enumerate}[label=\textnormal{(K\arabic*)}]
	\item $K_t^*$ is the adjoint of a linear continuous operator $K_t \colon H_t \to C(\olom)$,\label{K1i}
	\item $\nor{K_t} \leq C$ for some constant $C>0$ not depending on $t$,%
	\item the map $t \in [0,1] \mapsto K_t^* \rho \in H_t$ is strongly measurable for every fixed $\rho \in \M(\olom)$.\label{K3i}
        \end{enumerate}
\end{assumption}

Under \ref{K1i}--\ref{K3i}, \ref{H1i}--\ref{H3i} we have the following:  %
if $t \mapsto \rho_t \in \M(\olom)$ is narrowly continuous 
then the map $t \mapsto K_t^* \rho_t$ belongs to $\Ltwo$ (Lemma \ref{prop:good definition}). At this point, we are ready to rigorously define the regularization functional anticipated in \eqref{eq:intro_tikh}
as $J \colon \M \to [0,\infty]$, where
\begin{equation} \label{intro J}
J(\rho,m,\mu) := \frac{1}{2} \nor{K_t^* \rho_t - f_t}_{L^2}^2 + \alpha  B_\delta (\rho,m,\mu) + \beta \nor{\rho}_{\M(X)} \,,  
\end{equation}
if $(\rho,m,\mu) \in \mathcal{D}$ and $J:=\infty$ otherwise. The discrepancy term in $J$ is well defined since, if $(\rho,m,\mu) \in \mathcal{D}$ and $B_\delta (\rho,m,\mu)<\infty$, then $\rho=dt \otimes \rho_t$ with $t  \mapsto \rho_t$ narrowly continuous, so that $t \mapsto K_t^* \rho_t$ belongs to $\Ltwo$ (Proposition \ref{prop:well def}). %
Notice that, in addition to the regularizer $B_\delta(\rho,m,\mu)$, we also included $\nor{\rho}_{\M(X)}$ in the definition of $J$: This serves the purpose of enforcing weak* coercivity on $J$, since no initial data on $\rho$ is prescribed. %
 Our main theoretical results concerning existence of minimizers for $J$ and regularization properties are summarized in the following statements, which are contained in Theorem \ref{thm:existence} and Theorems \ref{thm:stability}, \ref{thm:vanishing}, respectively. 

\begin{theorem}\label{intro main thm}
Assume \ref{H1i}--\ref{H3i}, \ref{K1i}--\ref{K3i}. Let $f \in \Ltwo$, $\alpha, \beta >0$. Then %
$J$ admits a minimizer $(\rho,m,\mu) \in  \mathcal{D}$ satisfying $\rho \geq 0$, $\rho= dt \otimes \rho_t$ with $t \mapsto \rho_t$ narrowly continuous. If in addition the operators $K_t^*$ are injective for a.e.~$t \in [0,1]$, then the minimizer is unique.
\end{theorem}
In the next theorem, $J_{\alpha,\beta,f}$ denotes the functional $J$ in~\eqref{intro J} for  $\alpha, \beta > 0$ and $f \in L^2([0,1]; H)$.
\begin{theorem}[Regularization] \label{intro thm regularization}
Assume \ref{H1i}--\ref{H3i}, \ref{K1i}--\ref{K3i}. Let $f^\gamma, f^\dagger \in \Ltwo$ be noisy and exact data respectively, for noise level $\gamma>0$. 
\begin{enumerate}
\item[i)] Suppose that $f^n \to f^\gamma$ strongly in $L^2$, $\alpha,\beta>0$ and $(\rho^n,m^n,\mu^n) \in \argmin J_{\alpha,\beta,f^n}$., %
Then, up to subsequences, $(\rho^n,m^n,\mu^n)$ converges weakly* to $(\rho,m,\mu) \in \argmin J_{\alpha,\beta,f^\gamma}$.%
\item[ii)] Assume that $\nor{f^{\gamma_n} - f^\dagger}_{L^2} \leq \gamma_n$ and $\alpha_n, \beta_n \searrow 0$, such that $\gamma_n^2 /\min \{ \alpha_n, \beta_n\} \to 0$. If $(\rho^n,m^n,\mu^n) \in \argmin J_{\alpha_n,\beta_n,f^{\gamma_n}}$ then, up to subsequences, $(\rho^n,m^n,\mu^n)$ converges weakly* to $(\rho^\dagger,m^\dagger,\mu^\dagger) \in \mathcal{D}$ solving \eqref{intro inverse} and there exist $\alpha^*,\beta^* \in [1,\infty]$ such that
\[
(\rho^\dagger,m^\dagger,\mu^\dagger) \in \argmin  \,\, \alpha^* B_\delta (\rho,m,\mu) + \beta^* \nor{\rho}_{\M (X)} \,.
\]  
\end{enumerate}
\end{theorem}

\subsection{Application to dynamic MRI} \label{sec:intro:mri}
We apply the model \eqref{intro inverse} and its regularization \eqref{eq:intro_tikh} to undersampled dynamic MRI, yielding a reconstruction
approach via convex optimization which is principally capable of capturing
motion during the acquisition.
A common limiting factor to medical imaging techniques and MRI in particular
is acquisition speed such that, for instance,
data cannot be collected sufficiently fast in order to temporally resolve the beating heart or the lung while breathing. Consequently, static reconstruction approaches lead to severe artifacts. Thus, motion has to be taken into account by considering the dynamic setting in which at each time instance, data is severely undersampled and temporal data redundancies have to be exploited. For this purpose, we show that the optimal-transport regularization framework developed in this paper can be applied, leading to a regularizer that penalizes the displacements caused by motion and intrinsically recovers the motion field as well as the growth rate.

The forward problem in undersampled dynamic MRI in two dimensions is commonly stated as
follows: In each time instance $t$, the proton density $\rho_t$, a non-negative quantity,  needs to be
recovered from the measured data $f_t$. Taking coil sensitivities into account, $\rho_t$ and $f_t$ are linked via the Fourier transform. However, for each $t$, the Fourier
data is only acquired on subsets specified by the sampling strategy, leading to each $f_t$ generally living on a different subset of the so-called $k$-space and hence, being contained in a time-varying data Hilbert space $H_t$. Modelling the proton density
$\rho_t$ as a positive measure on the %
image domain $\Omega \subset \R^2$, denoting by $K_t^*$ an appropriately masked Fourier transform and considering
the unit time interval $[0,1]$,
the forward problem then indeed reads as $K_t^*\rho_t = f_t$ in $H_t$
for $t \in [0,1]$. This is made precise in the following.

Adopting the common model for parallel data acquisition 
(see, e.g., \cite{pruessmann,kcbus,kbpr,shsbs}),
let $\Om \subset \R^2$ be an open bounded domain representing the image domain and
let %
the complex coil sensitivities $c_j \in C_0(\R^2;\C)$ for $j=1,\dots,N$ with $N \geq 1$ to each of the $N$ receiver coils be given. The time-dependent sampling method is represented by a family of measures $\sigma_t \in \M^+ (\R^2)$ for $t \in [0,1]$. Such measures are required to satisfy some mild regularity assumptions, namely,
\begin{enumerate}[label=\textnormal{(M\arabic*)}]
\item $\nor{\sigma_t}_{\M(\R^2)} \leq C$ for a.e.~$t \in [0,1]$, where $C >0$ does not depend on $t$,\label{M1i}
\item the map $t \mapsto \int_{\R^2} \f(x) \, d\sigma_t(x)$ is measurable for each $\f \in C_0(\R^2;\C)$,\label{M2i}
\end{enumerate}
allowing for a variety of sampling methods, see Section~\ref{sec:MRI} for details.
The data space of measurements is then defined by $H_t:=L^2_{\sigma_t}(\R^2;\C^N)$, interpreted as a real Hilbert space and equipped with the norm $\norh{h}^2:=\sum_{j=1}^N \int_{\R^2} |h^j(x)|^2 \, d\sigma_t(x)$, where we denote $h=(h^1,\dots,h^N)$.
The forward operators  are given by
$K_t^* \colon \M(\olom) \to H_t$ defined via
	\[
	K_t^* \rho_t := (\F(c_1\rho_t), \dots, \F(c_N\rho_t)) \,,
 	\]
	where $\F$ is the Fourier transform and we interpret each $\F(c_j \rho_t)$ as
        an element of $L^2_{\sigma_t}(\R^2,\C)$.
        In Lemma \ref{MRI H} we show that under \ref{M1i}--\ref{M2i}, the spaces $H_t$ and the forward operators $K_t^*$ fulfill \ref{H1i}--\ref{H3i} and \ref{K1i}--\ref{K3i}, respectively. In this way the hypotheses of Theorems \ref{intro main thm}, \ref{intro thm regularization} are satisfied, and we can regularize the reconstruction problem \eqref{intro inverse} with the functional $J \colon \M \to [0,\infty]$ defined in \eqref{intro J}, which in this framework corresponds to 
	\[
	J(\rho,m,\mu):=\frac{1}{2} \sum_{j=1}^N \int_0^1 \nor{\F(c_j \rho_t)-f_t^j}_{L^2_{\sigma_t}(\R^2;\C)}^2   \, dt  + 
\alpha B_\delta (\rho,m,\mu) + \beta \nor{\rho}_{\M((0,1) \times \olom)} \,,
	\] 
where the measurements $f_t=(f_t^1,\dots,f_t^N)$ belong to $\Ltwo$. 
This shows in particular that optimal-transport regularization for undersampled
dynamic MRI leads to well-posed convex optimization problems. These are accessible
to analysis as well as efficient and stable numerical minimization algorithms.
Section~\ref{sec:conclusions} %
provides some perspectives
for the latter. %

\section{Dynamic optimal transport} \label{sec:optimal}

The aim of this section is to provide the essential elements to define the optimal transport regularizer appearing in \eqref{intro J}. We refer the reader to Appendix \ref{app:measure:prel} for measure theory definitions and results which will be needed in the following. Throughout the section, $\Om \subset \R^d$ is an open bounded domain, with $d \in \N$, $d \geq 1$ and we set $X:=(0,1) \times \olom$ to be the time-space cylinder. We also define the space $\M := \M(X) \times \M(X;\R^d)\times \M(X )$.  
In Section \ref{sec:continuity equation} we introduce the concept of measure solution to the continuity equation with source
\begin{equation}  \label{sec2 cont}
\de_t \rho + \Div m = \mu    \,\, \text{ in } \,\, X \,, 
\end{equation}
where $(\rho,m,\mu) \in \M$. Here $\rho$ represents a density, $m$ a momentum field advecting $\rho$ and $\mu$ a source term, accounting for local mass change. We then investigate properties of solutions $\rho \in \M (X)$ of \eqref{sec2 cont}. In particular in Proposition \ref{prop:dis} we show that positive solutions to \eqref{sec2 cont} disintegrate as $\rho = dt \otimes \rho_t$ with $\{\rho_t\}_{t \in [0,1]}$ Borel family of positive measures over $\olom$. In Proposition \ref{prop:cont rep} we prove that, under some growth assumptions on $m$ and $\mu$, the curve $t \mapsto \rho_t$ is actually narrowly continuous. 
Finally, in Section \ref{sec:energy} we introduce the optimal transport energy $B_\delta$ at \eqref{intro J}, and list some of its properties in Proposition \ref{prop:bb prop}.

\subsection{Continuity equation} \label{sec:continuity equation}

We want to consider measure valued (distributional) solutions to the continuity equation \eqref{sec2 cont}
with suitable boundary conditions. The precise definition is as follows.

\begin{definition} \label{def:meas sol}
We say that $(\rho,m,\mu) \in \M$ is a \textit{measure solution} to \eqref{sec2 cont} if
\begin{equation} \label{cont weak}
\int_{X}  \de_t \f \, d\rho + 
\int_{X}  \nabla \f \cdot  dm +
\int_{X}  \f \,  d\mu  = 0   \quad \text{for all} \quad \f \in C^{\infty}_c ( X ) \,.
\end{equation}
\end{definition}

We remark that the above weak formulation includes zero flux boundary conditions for the momentum $m$ on $\de \Om$, and no initial and final data is prescribed on $\rho$. Moreover one can test \eqref{cont weak} against functions in $C^1_c ( X )$ (see \cite[Remark 8.1.1]{ags}). In the following proposition we show that positive solutions to  \eqref{sec2 cont} can be disintegrated with respect to the Lebesgue measure $dt$ on $(0,1)$ (see Section \ref{sec:disintegration} for details on disintegration). To this end, let $\pi \colon X \to (0,1)$ be the projection on the time coordinate.

\begin{proposition} \label{prop:dis}
Assume that $(\rho,m,\mu) \in \M$ satisfies \eqref{cont weak}, with $\rho \in \M^+(X)$.  Then $\rho$ disintegrates, with respect to $dt$, as $\rho =dt \otimes  \rho_t$, where $\rho_t \in \M^+(\overline{\Om})$ for a.e.~$t$. Moreover $t \mapsto \rho_t(\overline{\Om})$ is a function of bounded variation, with distributional derivative $\pi_{\#} \mu$. In particular, if the source $\mu = 0$, then the total mass $\rho_t(\overline{\Om})$ is constant in time.
\end{proposition}

\begin{proof}
In order to apply Theorem \ref{thm:disint} we need to show that $\pi_\# \rho \ll dt$. Let $\tilde{\f} \in C^\infty_c ((0,1))$ and define $\f:=\tilde{\f} \circ \pi \in C^\infty_c(X)$. By plugging $\f$ in \eqref{cont weak} we get
$\int_0^1 \tilde{\f}' \, d (\pi_\# \rho) = -\int_0^1 \tilde{\f} \, d (\pi_\# \mu)$,
so that $(\pi_\# \rho)'=\pi_\# \mu$ in the sense of distributions. Since $\pi_\# \mu \in \M((0,1))$, there exists $u \in BV((0,1))$ such that $\pi_\# \mu=u'$. Therefore $\pi_\# \rho \ll dt$ and there exists a Borel family $\rho_t \in \M(\overline{\Om})$ such that $\rho=dt \otimes \rho_t$.
In particular, since $\rho_t$ is a Borel family, the map $t \mapsto \rho_t(\overline{\Om})$ is measurable. Moreover it belongs to $L^1((0,1))$, since 
$\int_0^1 |\rho_t(\overline{\Om})| \, dt  = \rho(X)$
which is finite by assumption. Finally notice that $\pi_\#(dt \otimes \rho_t) = \rho_t(\overline{\Om}) \, dt$, which together with $(\pi_\# \rho)'=\pi_\# \mu$ implies that $t \mapsto \rho_t(\overline{\Om})$ belongs to $BV((0,1))$, with distributional derivative given by $\pi_\# \mu$. 
\end{proof}

\begin{definition}
A curve $t \in [0,1] \mapsto \rho_t \in \M(\olom)$ is \textit{narrowly continuous} if the map
\[
t \mapsto \int_{\olom} \f(x) \, d\rho_{t}(x) 
\]
is continuous for every fixed $\f \in C(\overline{\Om})$. 
We denote by $\curves$ the set of such curves, and by $\pcurves$ the set of narrowly continuous curves of positive measures.
\end{definition}

In the next proposition we show that if $(\rho,m,\mu)$ solves the continuity equation with appropriate energy bounds, the disintegration measures $\rho_t$ are defined for every $t$ and are narrowly continuous. This is a well-known result for $\mu =0$.  For completeness we will carry out the proof in Appendix~\ref{app:meas:narrow}, by adapting the argument used to prove the homogeneous version (see Lemma 8.1.2 in \cite{ags}).

\begin{proposition}[Continuous representative] \label{prop:cont rep}
Let $(\rho,m,\mu) \in \M$ be a solution of \eqref{cont weak}, with $\rho \in \M^+(X)$. Let $\rho_t \in \M^+(\olom)$ be the disintegration of $\rho$ with respect to $dt$. Assume that $m=dt \otimes v_t \rho_t$ and  $\mu = dt \otimes g_t \rho_t$ with $v_t \colon X \to \R^d$, $g_t \colon X \to \R$ measurable functions such that
\begin{equation} \label{cont int}
\int_0^1 \int_{\olom} |v_t(x)| \, d\rho_t (x) \, dt < \infty  \quad \text{ and } \quad \int_0^1 \int_{\olom} |g_t(x)| \, d\rho_t (x) \, dt < \infty \,.
\end{equation}
 Then there exists a narrowly continuous curve $(t \mapsto \tilde{\rho}_t) \in \pcurves$ such that $\rho_t = \tilde{\rho}_t$ a.e.~in $(0,1)$. Moreover for each $\f \in C^1_c([0,1]\times \overline{\Om})$ and $0 \leq t_1 \leq t_2 \leq 1$ we have
\begin{equation} \label{cont boundary}
\int_{t_1}^{t_2} \int_{\olom} (\de_t \f +  \nabla \f \cdot v_t + \f  g_t) \, d\rho_t(x) \, dt = 
\int_{\olom} \f(t_2,x) \, d\tilde{\rho}_{t_2} (x) - 
\int_{\olom} \f(t_1,x) \, d\tilde{\rho}_{t_1} (x) \,.
\end{equation}
\end{proposition}

In the rest of the paper we will identify $\rho_t$ with its narrowly continuous representative $\tilde{\rho}_t$ whenever the assumptions of Proposition \ref{prop:cont rep} hold, and use the notation $\rho_t$.

\subsection{Optimal transport energy} \label{sec:energy}

We want to introduce the Wasserstein--Fisher--Rao energy (or Hellinger--Kantorovich) as originally done in \cite{chizat2,chizat,liero,liero2,kmv}. First, define the convex set 
\begin{equation} \label{keydelta}
K_\delta :=\left\{ (a,b,c) \in \R \times \R^d \times \R \, \colon \, a+ \frac{1}{2} \left(|b|^2 +\frac{c^2}{\delta^2}\right) \leq 0 \right\} \,,
\end{equation}          
with $\delta \in (0,\infty]$ fixed parameter and $\frac{c^2}{\infty} = 0$ for every $c \in \R$.
For $(t,x,y) \in \R \times \R^d \times \R$ set
\[%
\Psi_\delta (t,x,y):= \begin{cases}
 \frac{|x|^2 +\delta^2 y^2}{2t}    & \text{if } t >0 \,,\\
 0                   & \text{if } t=|x|=y=0 \,, \\
 \infty	         & \text{otherwise } \,,
 \end{cases}
\]%
where $\infty y^2 = \infty$ for $y \neq 0$ and $\infty y^2 = 0$ for $y=0$.
We have that $\Psi_\delta$ is the Legendre conjugate of the characteristic function $\rchi_{K_\delta}$ \cite{chizat}, that is,
\[
\Psi_\delta (t,x,y) = \sup_{(a,b,c) \in K_\delta}  (at + b \cdot x + cy) \quad \text{ for each } \quad (t,x,y) \in \R \times \R^d \times \R \,.
\]
In particular $f$ is convex, lower semicontinuous and 1-homogeneous.

\begin{definition}[Transport energy]
	Let $(\rho,m,\mu) \in \mathcal{M}$. We define the \textit{transport energy} as
\begin{equation} \label{bb}
B_\delta (\rho,m,\mu) := \sup \left\{ \int_X a \, d \rho + \int_{X} b  \cdot dm + \int_X c  \,d\mu \,\, \colon \, (a,b,c) \in C_0 (X;K_\delta) \right\} \,. 
\end{equation}
\end{definition}

We summarize some of the properties of the functional $B_\delta$ that will be needed throughout this paper. The proof is omitted, and it can be easily adapted from the one in \cite[Prop 5.18]{sant}.

\begin{proposition} \label{prop:bb prop}
The functional $B_\delta$ defined in \eqref{bb} is convex and lower semicontinuous for the weak* convergence. Moreover it satisfies the following properties:   
\begin{enumerate}[(i)]
\item $B_\delta (\rho,m,\mu) \geq 0$, 	
\item assume that $\rho,m,\mu \ll \lambda$ for some $\lambda \in \M^+(X)$. Then
\[
B_\delta (\rho,m,\mu)= \int_X \Psi_\delta \left( \frac{d\rho}{d\lambda}, \frac{dm}{d\lambda},\frac{d\mu}{d\lambda}\right) \, d \lambda \,,
\]
\item if $B_\delta (\rho,m,\mu)< \infty$ then $\rho \geq 0$ and $m,\mu \ll \rho$, 
\item if $\rho \geq 0$ and $m,\mu \ll \rho$, 
then $m = v \rho$, $\mu=g \rho$ for measurable $v \colon X \to \R^d$, $g \colon X \to \R$ and 
\[
B_\delta (\rho,m,\mu)= \int_{X} \Psi_\delta (1,v,g)  \, d\rho=\frac{1}{2} \int_{X} \left(|v|^2 + \delta^2 g^2 \right) \, d\rho \,.
\]
\end{enumerate}
\end{proposition}

\section{Time dependent Bochner spaces} \label{sec:hilbert setting}

In this section we construct a class of Bochner spaces of Hilbert spaces valued functions, where the Hilbert space can vary in time. Here the underlying measure space is the unit interval with the Lebesgue measure. This can however be easily generalized to arbitrary measure spaces. Moreover a generalization to Banach spaces valued functions seems possible, however, it is out of the scope of this paper.  More precisely,  we want to define a concept of integrability for functions
$f \colon [0,1] \to \{H_t\}_t$, where $H_t$ is a Hilbert space for each time $t$, and $f(t) \in H_t$ for all $t$. In order to do that we will closely follow the approach to define classic Bochner spaces (see \cite[Ch II]{diestel}, \cite[Ch 11]{aliprantis}). 
In Section \ref{sec:func_sett} we establish the functional analytic setting and assumptions under which we carry out the construction. In Section \ref{sec:meas} we define suitable notions of measurability and provide the equivalent of the classic Pettis measurability theorem (see Theorem \ref{thm:pettis}). Such result is instrumental to the following analysis, as it provides a practical characterization of strong measurability. In Section \ref{sec:int} we define integrability for maps $f \colon [0,1] \to H$ and characterize it in Theorem \ref{thm:int}. We then proceed to define the time dependent Bochner spaces $L^p([0,1];H)$. %
Notice that, in contrast to the classic Bochner theory, we will not define a notion of integral for integrable maps $f \colon [0,1] \to H$, but only of integrability. However, a comparison with the classical theory is possible, and it will be carried out in Appendix \ref{app:bochner:comparison}.

\subsection{Functional setting}   \label{sec:func_sett}
Let $\{H_t\}$ for $t \in [0,1]$ be a family of real Hilbert spaces with norms and scalar products denoted by $\norh{\cdot}$ and $\inner{\cdot}{\cdot}$, respectively. 
The interval $[0,1]$ is equipped with the Lebesgue measure. As usual, we denote by $|E|$ the measure of a set $E \subset [0,1]$ and by $\rchi_E$ its characteristic function, defined as $\rchi_E (t) := 1$ if $t \in E$ and $\rchi_E(t):=0$ otherwise. 
Let $H:=\cup_{t \in [0,1]} H_t$. We will denote by $f \colon [0,1] \to H$ maps such that $f(t) \in H_t$ for a.e.~$t \in [0,1]$. 
  Let $D$ be a real Banach space with norm denoted by $\nor{\cdot}_D$ and duality by $\inn{\cdot}{\cdot}_{D^*,D}$. Assume that for a.e.~$t \in [0,1]$ there exists a linear continuous operator $i_t \colon D \to H_t$ with the following properties: 
  \begin{enumerate}[label=\textnormal{(H\arabic*)}]
    \item	$\nor{i_t} \leq C$ for some constant $C>0$ not depending on $t$, \label{H1}
    \item $i_t (D)$ is dense in $H_t$, \label{H2}
    \item the map \label{H3}
  $
  t \mapsto \inner{i_t \f }{i_t \psi}
  $
is Lebesgue measurable for every fixed $\f,\psi \in D$. 
  \end{enumerate}
  The adjoint of $i_t$ is $i_t^* \colon H_t \to D^*$, defined by 
  $\inn{i_t^* h}{\f}_{D^*,D}:=\inner{h}{i_t \f}$ for all $h \in H_t, \f \in D$
   (here we identified $H_t$ with its dual).
  Notice that from \ref{H1}  it follows that $i_t^* \colon H_t \to D^*$ is linear continuous and such that $\nor{i_t^*} \leq C$.  Moreover from \ref{H2} we have that $i_t^*$ is injective. Throughout the section, we say that $g=f$ if the equality holds a.e.~in $[0,1]$. Moreover we say that $f_n \to f$ a.e.~if $\lim_n \norh{f_n(t) - f(t)}=0$ for a.e.~$t\in [0,1]$.

\subsection{Measurability in time dependent spaces} \label{sec:meas}

In this section we introduce suitable measurability notions for maps $f \colon [0,1]\to H$, and prove our version of Pettis' Theorem. We refer the reader to \cite[Ch II]{diestel} for classic measurability definitions. 

\begin{definition}[Step function] \label{def:step}
A map $f \colon [0,1] \to D$ is a \textit{step function} if  $f=\sum_{j=1}^N \rchi_{E_j} \f_j$ with $N \in \N$, $\f_j \in D$ and $E_j$ Lebesgue measurable pairwise disjoint subsets of $[0,1]$. 
\end{definition}

\begin{definition}[Measurability] \label{def:strong}
Let $f \colon [0,1] \to H$. We say that
\begin{enumerate}
\item[i)] $f$ is \textit{strongly measurable} if there exists a sequence of step functions $f_n$ such that
\[
\lim_n \norh{i_t f_n (t) - f(t)} = 0 \,\,\, \text{ for a.e. } \,\, t \in [0,1]\,,
\]
\item[ii)] $f$ is \textit{weakly measurable} if 
$t \mapsto \inner{i_t \f}{f(t)}$	
is Lebesgue measurable for each $\f \in D$,
\item[iii)]	 $f$ is \textit{essentially separably valued} if there exist a measurable set $E \subset [0,1]$ with $|E|=0$ and a countable subset $S \subset D$ with the following property: for every $\e >0$ and $t \in [0,1] \smallsetminus E$, there exists an element $\f \in S$ such that
	\[
	\norh{i_t \f - f(t)} < \e \,.
	\]  
\end{enumerate}
\end{definition}

Notice that, if $H_t = H$ for each $t \in [0,1]$, with $H$ fixed Hilbert space, $D = H$ and $i_t \f := \f$, then Definitions \ref{def:step} and \ref{def:strong} are equivalent to the classic ones given in Chapter II of \cite{diestel}.

\begin{remark} \label{rem1}
Let $p \geq 1$ and $f=\sum_{j=1}^N \rchi_{E_j} \f_j$ be a step function. Then the map $t \to \norh{i_t f(t)}^p$ is measurable and $\int_0^1 \norh{i_t f(t)}^p \, dt<\infty$. Indeed, 
$
\norh{i_t f(t)}^2 = \sum_{j=1}^N \rchi_{E_j}(t) \inner{i_t \f_j}{i_t \f_j}$,	
so that $t \mapsto \norh{i_t f(t)}^2$
is measurable by \ref{H3}, and hence also $t \mapsto \norh{i_t f(t)}^p$ is. Moreover by \ref{H1} 
\[
\int_0^1 \norh{i_t f(t)}^p \, dt = \sum_{j=1}^N \int_{E_j} \norh{i_t \f_j}^p \, dt\leq C^p \sum_{j=1}^N |E_j|   \nor{\f_j}_D^p < \infty \,.
\]
\end{remark}

\begin{remark} \label{rem:strong}
It is easy to check that strong measurability is stable under sums, scalar multiplication and pointwise a.e.~convergence.
Moreover if $f \colon [0,1] \to H$ is strongly measurable then the map $t \mapsto \norh{f(t)}$ is Lebesgue measurable, since $f$ can be approximated a.e. by step functions $f_n$ and $t \mapsto \norh{i_t f_n (t)}$ is measurable for every fixed $n$ by Remark \ref{rem1}.
\end{remark}

The above definitions are linked together by the analogous of the classic Pettis measurability Theorem (see \cite[Ch II.1, Thm 2]{diestel}).

\begin{theorem}[Pettis] \label{thm:pettis}
	Let $f \colon [0,1] \to H$. Then $f$ is strongly measurable if and only if $f$ is weakly measurable and essentially separably valued.
\end{theorem}
For a proof of this theorem, see Appendix \ref{app:bochner:auxiliary}.
By inspecting the proof, one can see that the following corollary holds.

\begin{corollary} \label{cor1}
Let $f \colon [0,1] \to H$. Then $f$ is strongly measurable if and only if it is the a.e.~uniform limit of a sequence of countably valued functions $f_n \colon [0,1] \to D$, that is, if
\[
\lim_n \norh{ i_t f_n (t) - f(t) } = 0     
\]	
uniformly for a.e.~$t \in [0,1]$, for some $f_n = \sum_{j=1}^\infty \rchi_{E_{n,j}} \f_{n_j} $ with $\f_{n_j} \in D$ and $\{E_{n,j}\}_{j \in \N}$ measurable and pairwise disjoint subsets of $[0,1]$.
\end{corollary}

\begin{proposition}[Separable case] \label{rem:sep} 
Assume that $D$ is separable. Then strong measurability is equivalent to weak measurability.	
\end{proposition}
\begin{proof}
Let $S=\{\f_n\} \subset D$ be countable and dense. Then $\{i_t \f_n\}$ is dense in $H_t$: indeed fix $\e >0$ and $h \in H_t$. By \ref{H2} there exists $\f \in D$ such that $\norh{h-i_t \f} < \e /2$. Let $\f_n$ be such that $\nor{\f - \f_n}_D < \e / 2C$ where $C$ is the constant in \ref{H1}. Then $\norh{h-i_t \f_n} < \e$ by triangle inequality and \ref{H1}. Therefore $H_t$ is separable and it is immediate to check that in this case every function $f \colon [0,1] \to H$ is essentially separably valued. By Theorem \ref{thm:pettis} the thesis follows.
\end{proof}

\subsection{\texorpdfstring{Integration and $L^p$ spaces}{Integration and Lp spaces}} \label{sec:int}

\begin{definition}[Integrability]\label{def:int}
Let $f \colon [0,1] \to H$ be strongly measurable according to Definition \ref{def:strong}. We say that $f$ is \textit{integrable} if there exists a sequence $\{f_n\}$ of step functions $f_n \colon [0,1] \to D$ such that
\begin{equation} \label{integrable}
\lim_n \int_0^1 \norh{i_t f_n(t) - f(t)} \,dt = 0 \,.
\end{equation}
\end{definition}

Notice that the definition is well posed, since the map $t \to \norh{i_t f_n(t) - f(t)}$ is Lebesgue measurable for each fixed $n$ (see Remark \ref{rem:strong}), hence its integral is well defined. Analogously to the classic case (\cite[Ch II.2, Thm 2]{diestel}), we can characterize integrability as stated in the following theorem.

\begin{theorem}[Characterization of integrability] \label{thm:int}
Let $f \colon [0,1] \to H$ be strongly measurable. Then 
 	$f$ is integrable if and only if 
 	$\int_0^1 \norh{f(t)} \, dt < \infty$.
\end{theorem}

\begin{proof}
Assume that $f$ is integrable. By \eqref{integrable}, for sufficiently large $N$ we have
\[
\int_0^1 \norh{f(t)} \, dt \leq \int_0^1 \norh{i_t f_N (t) - f(t)} \, dt + 
\int_0^1 \norh{i_t f_N (t)} \, dt < 1 + \int_0^1 \norh{i_t f_N (t)} \, dt
\]
and the thesis follows by Remark \ref{rem1}, since $f_N$ is a step function. Conversely, assume that $\int_0^1 \norh{f(t)} \, dt < \infty$. Since $f$ is strongly measurable, by Corollary \ref{cor1}
	there exists a sequence $\{g_n\}$ of countably valued maps $g_n \colon [0,1] \to D$ such that
$
\norh{ i_t g_n (t) - f(t) } < 1/n
$	
for a.e.~$t \in [0,1]$. In particular $\norh{ i_t g_n (t)} \leq  \norh{f(t) } + 1/n$, so that $\int_0^1 \norh{ i_t g_n (t) } \, dt < \infty$. By construction 
$
g_n(t)=\sum_{j=1}^{\infty} \rchi_{E_{n,j}} \f_{n,j} 
$
with sets $E_{n,j} \subset [0,1]$ measurable and pairwise disjoint. Hence there exists a sequence $\{k_n\}$ in $\N$ such that 
$\sum_{j=k_n +1}^{\infty} \int_{E_{n,j}} \norh{i_t g_n(t)} \, dt < n^{-1}$.
Therefore by setting $f_n(t) := \sum_{j=1}^{k_n}  \rchi_{E_{n,j}} \f_{n,j}$ we obtain
\[
\int_0^1 \norh{i_t f_n(t) - f(t)} \, dt \leq 
\int_0^1 \norh{i_t g_n(t) - f(t)}\, dt + \int_0^1 \norh{i_t f_n(t) - i_t g_n(t)} \, dt < \frac{2}{n}\,,
\]
proving that $f$ is integrable. 
\end{proof}

As a corollary of the above theorem, we obtain that a suitable version of Lebesgue's dominated convergence theorem holds in our setting. We postpone its proof to Appendix \ref{app:bochner:auxiliary}.

\begin{theorem}[Dominated convergence] \label{thm:dominated_convergence}
Let $f_n \colon [0,1] \to H$ be a sequence of integrable functions such that
$f_n \to f$ a.e.~in $[0,1]$ and that there exists $g \in L^1([0,1])$ satisfying
$\norh{f_n(t)} \leq g(t)$ a.e.~in $[0,1]$. Then $f$ is integrable and 
$f_n \to f$ strongly in $L^1([0,1];H)$.
\end{theorem}

\begin{definition}[$L^p$ space]
Fix $1 \leq p < \infty$. We define the space of the $p$-integrable functions
\begin{equation} \label{L2}
\Lp := \left\{ f \colon [0,1]  \to H \, \colon \,  
  f \, \text{ strongly measurable }, \, \int_0^1 \norh{f(t)}^p \, dt < \infty \right\} \,.
\end{equation}
\end{definition}

In \eqref{L2} we identify functions that coincide almost everywhere. 
Notice that \eqref{L2} is well posed: Indeed since $f$ is strongly measurable, then the map $t \mapsto \norh{f(t)}^p$ is measurable (Remark \ref{rem:strong}), and hence $\int_0^1 \norh{f(t)}^p \, dt $ is well defined (possibly infinite). 

\begin{remark}
If $H_t \equiv H$ with $H$ fixed Hilbert space and $D=H$, then $\Lp$ coincides with the respective classic Bochner spaces (see \cite[Ch II]{diestel}).  
\end{remark}

\begin{theorem} \label{thm:completeness}
	Let $1 \leq p < \infty$. We have that $\Lp$ is a Banach space with the norm $\nor{f}_{L^p}^p := \int_0^1 \norh{f(t)}^p \, dt$. %
Moreover $\Ltwo$ is a Hilbert space with the inner product  $\innerL{f}{g}:=\int_0^1 \inner{f(t)}{g(t)} \, dt$. 
\end{theorem}

The proof of the above theorem is postponed to Appendix~\ref{app:bochner:auxiliary}.

\begin{remark}[$p=\infty$]
It is is possible to treat the case $p=\infty$ by defining $L^\infty ([0,1];H)$ as the set of strongly measurable functions $f \colon [0,1]  \to H$ such that $\esssup_{t \in [0,1]} \norh{f(t)} < \infty$. By adapting the proof for classical Bochner spaces, one can show that $L^\infty ([0,1];H)$ is a Banach space with the norm $\nor{f}_{L^\infty}:= \esssup_{t \in [0,1]} \norh{f(t)}$.
\end{remark}

\begin{remark}[Dual spaces]
The space $\Ltwo$ is self-dual, being a Hilbert space. We believe that for any $p\geq 1$ one has the isometry $L^p([0,1];H)^*=L^q([0,1];H)$ with $1/p+1/q=1$. %
\end{remark}

\begin{example}[Narrowly continuous curves]
Let $\Om \subset \R^d$ be an open bounded domain with $d \in \N$, $d \geq 1$.
Let $\rho_t \in \pcurves$, $D:=C(\overline{\Om})$ with the supremum norm, $H_t:=L^2_{\rho_t}(\Om)$ with the norm $\norh{h}^2:=\int_\Om |h(x)|^2 \, d\rho_t (x)$. Define $i_t \colon D \to H_t$ by $i_t \f := \f$. 
It is left to the reader to check that \ref{H1}--\ref{H3} are satisfied.
Therefore we can define the space $L^p([0,1];H)$ for each $p \geq 1$. Notice that in this case $L^2([0,1];H)$ is isometric to $L^2_{\rho}([0,1] \times \olom)$, where $\rho:=dt \otimes \rho_t$. 
\end{example}

Finally, we would like to mention that, although we do not define a notion of integral for maps in $L^p([0,1];H)$, a comparison to the classical Bochner theory is still possible. Indeed, if $f \colon [0,1] \to H$, then by definition $i_t^*f \colon [0,1] \to D^*$ and the codomain is a fixed Banach space. In Appendix~\ref{app:bochner:comparison} we show that, assuming $f \in L^1([0,1];H)$, we always have that $i_t^* f$ is weakly* integrable (see Proposition \ref{prop:gelfand}). However, Bochner integrability fails in general, as shown in Example \ref{ex2}. Nevertheless, under suitable additional assumptions, one can show that Bochner integrability can be guaranteed (Proposition \ref{prop:bochner}).

\section{Regularization of dynamic inverse problems} \label{sec:regularization}

In this section we define and study the properties of the optimal transport based functional \eqref{intro J}, and we establish it as a regularizer for the dynamic inverse problem \eqref{intro inverse}. Throughout the section, the functional analytic setting will be the following. 
Let $\Om \subset \R^d$ be an open bounded domain, where $d \in \N$, $d \geq 1$, and define again the time space domain $X:=(0,1) \times \olom$. Let $\{H_t\}$ be a family of Hilbert spaces for $t \in [0,1]$, $D$ a Banach space and $i_t \colon D \to H_t$ linear operators which satisfy the assumptions \ref{H1}--\ref{H3} of Section \ref{sec:func_sett}. Assume given a family of linear continuous operators $K^*_t \colon \M (\overline{\Om}) \to H_t$ such that, for a.e.~$t \in [0,1]$, 
\begin{enumerate}[label=\textnormal{(K\arabic*)}]
\item $K^*_t$ is the adjoint of a linear continuous operator $K_t \colon H_t \to C(\overline{\Om})$,  \label{K1} 
\item[(K1')] 
\makeatletter\protected@edef\@currentlabel{(K1')}%
\makeatother%
$K_t^*$ is weak*-to-weak continuous,\label{K1bis}	
\item $\nor{K_t^*} \leq C$ for some constant $C>0$ that does not depend on $t$, \label{K2}
\item  \label{K3} the map $t \mapsto K^*_t \rho$ is strongly measurable in the sense of Definition \ref{def:strong}, for all $\rho \in \M(\olom)$. 
\end{enumerate}
We remark that conditions \ref{K1} and \ref{K1bis} are equivalent. 
As before, the space of narrowly continuous curves with values in the measures and in the positive measures are denoted by $\curves$ and $\pcurves$ respectively. The dynamic inverse problem we aim to regularize is the following: Given some data $f \in L^2([0,1];H)$, find a narrowly continuous curve $t \mapsto \rho_t$ such that
\begin{equation} \label{dyn inv}
K^*_t \rho_t = f_t \,,  \,\, \text{for a.e.} \,\, t \in [0,1].
\end{equation}
We regularize \eqref{dyn inv} as follows. Let $\M$ be defined by
\[
\M := \M (X) \times \M (X;\R^d) \times  \M (X) \,,
\]
and introduce the convex linear space of triples in $\M$ satisfying the continuity equation 
\[
\mathcal{D}:=\left\{ (\rho,m,\mu) \in \M \, \colon \, \de_t \rho+\Div m=\mu \, \text{ in the sense of } \, \eqref{cont weak} \right\}\,.
\]

\begin{definition}[Regularized problem]
	Let $f \in \Ltwo$ and $\alpha,\beta>0$, $\delta \in (0,\infty]$. The regularizer of \eqref{dyn inv} is the functional $J \colon \M \to [0,\infty]$ defined by
	\begin{equation} \label{def:J}
	J(\rho,m,\mu) :=
	 \frac{1}{2} \int_0^1 \norh{K^*_t \rho_t-f_t}^2 \, dt + \alpha B_\delta(\rho,m,\mu) + \beta  \nor{\rho}_{\M(X)}      
	\end{equation}
if $(\rho,m,\mu ) \in \mathcal{D}$ and $J=\infty$ otherwise. Here $\rho=dt \otimes \rho_t$ is the disintegration of $\rho$ with respect to time and $B_\delta$ is the transport energy defined in \eqref{bb}. 
\end{definition}

We will proceed as follows. First, in Section \ref{sec:reg} we show that the inverse problem in \eqref{dyn inv} and the functional $J$ in \eqref{def:J} are well defined, in the following sense: Given a triple $(\rho,m,\mu) \in \mathcal{D}$ with finite transport energy $B_\delta (\rho,m,\mu)$, then $\rho \geq 0$ and it disintegrates into $\rho=dt \otimes \rho_t$ with $t \mapsto \rho_t$ narrowly continuous. For such curves, we show in Lemma \ref{prop:good definition}, that $t \mapsto K_t^* \rho_t$ is a measurement belonging to $\Ltwo$, providing well-definition for \eqref{dyn inv}--\eqref{def:J}. In Section \ref{sec:existence} we show that
\begin{equation} \label{sec 4 min}
\inf_{(\rho,m,\mu) \in \mathcal{D}} J(\rho,m,\mu)
\end{equation}
admits at least one solution, and the minimizer is unique under additional assumptions on the operators $K_t^*$. This will be the content of Theorem \ref{thm:existence}. Finally, in Section \ref{sec:stability}, we investigate stability of the solutions to \eqref{sec 4 min} and convergence for vanishing noise level. %

\subsection{Well-definition} \label{sec:reg}

In this section we want to show that the definition of the functional $J$ at \eqref{def:J} is well posed. The first step is to ensure that the fidelity term is well defined, namely, that given $\rho_t \in \curves$ then the map $t \in [0,1] \mapsto K_t^* \rho_t \in H_t$ belongs to $\Ltwo$. This fact will be established in the following Lemma. 
 
\begin{lemma} \label{prop:good definition}
If $\rho_t \in \curves$ then $(t \mapsto K_t^* \rho_t) \in \Ltwo$.	
\end{lemma}

Let us postpone the proof for a moment. As a consequence of the Lemma we have the following.

\begin{proposition} \label{prop:well def}
The functional $J$ at \eqref{def:J} is well defined.
\end{proposition}

\begin{proof}[Proof of Lemma \ref{prop:good definition}]
\smallskip
\noindent \textbf{Part 1.}
Let $\rho_t \in \curves$. First we show that the map $t \mapsto K_t^* \rho_t$ is strongly measurable according to Definition \ref{def:strong}. We do so by means of Theorem \ref{thm:pettis}, by proving that $t \mapsto K_t^* \rho_t$ is weakly measurable and essentially separably valued. 

\smallskip
\noindent \textit{Claim 1:}
the map $t  \mapsto K_t^* \rho_t$ is weakly measurable as per Definition \ref{def:strong}, that is, $t \mapsto \inner{K_t^* \rho_t}{i_t \f}$ 	
is measurable for every fixed $\f \in D$.

\smallskip
\noindent \textit{Proof of Claim 1.}
By definition and \ref{K1} we have
$
\inner{K_t^* \rho_t}{i_t \f}=\inn{\rho_t}{K_t i_t \f}_{\M(\overline{\Om}),C(\overline{\Om})} \,.
$
Notice that the map $t \in [0,1] \mapsto K_t i_t \f \in C(\overline{\Om})$ is strongly measurable in the classic sense (\cite[Ch II]{diestel}). To see this, since $C(\overline{\Om})$ is separable, by the classic Pettis Theorem (\cite[Ch II.1, Thm 2]{diestel}), it is enough to prove that $t  \mapsto K_t i_t \f$ is weakly measurable, meaning that 
$
t \mapsto \inn{\rho}{K_t i_t \f}_{\M(\overline{\Om}),C(\overline{\Om})}
$  
is measurable for every fixed $\rho \in \M(\overline{\Om})$. The latter holds because 
$
\inn{\rho}{K_t i_t \f}_{\M(\overline{\Om}),C(\overline{\Om})} = \inner{K_t^* \rho}{i_t \f}
$ 
and the map $t \mapsto K^*_t \rho$ is strongly measurable by assumption \ref{K3}, and hence weakly measurable by Theorem \ref{thm:pettis}. By definition of classic strong measurability, there exists a sequence $\{f_n\}$ of step functions $f_n \colon [0,1] \to C(\overline{\Om})$, such that $f_n(t)=\sum_{j=1}^{N_n} \rchi_{E_{j,n}} f_{j,n}$ with $\{E_{j,n}\}_{j=1}^{N_n}$ measurable partition of $[0,1]$, $f_{j,n} \in C(\overline{\Om})$, and such that 
\begin{equation} \label{claim111}
\lim_{n \to \infty} \nor{f_n(t) - K_t i_t \f}_{C(\overline{\Om})} = 0 
\end{equation}
for a.e.~$t \in [0,1]$. We have that the map 
$t \mapsto \inn{\rho_t}{f_n(t)}_{\M(\overline{\Om}),C(\overline{\Om})}$
is measurable for each fixed $n \in \N$, since 
$
\inn{\rho_t}{f_n(t)}_{\M(\overline{\Om}),C(\overline{\Om})} = 
\sum_{j=1}^{N_n} \rchi_{E_{j,n}} \int_{\olom} f_{j,n} \, d\rho_t
$
and the maps $t \mapsto \int_{\olom} f_{j,n}  \, d\rho_t $ are continuous by narrow continuity of $t \mapsto \rho_t$. By Proposition \ref{prop:ub} we have that $\sup_{t \in [0,1]}\nor{\rho_t}_{\M (\overline{\Om})}< \infty$. Combining this with \eqref{claim111} yields
\[
|\inn{\rho_t}{f_n(t) - K_t i_t \f}_{\M(\overline{\Om}),C(\overline{\Om})} | \leq C \nor{f_n(t) - K_t i_t \f}_{C(\overline{\Om})} \to 0
\]
as $n \to \infty$, for a.e.~$t \in [0,1]$. Hence $t \mapsto \inn{\rho_t}{K_t i_t \f}_{\M(\overline{\Om}),C(\overline{\Om})}$ is measurable, being the a.e.~limit of measurable maps, and the claim follows.

\smallskip
\noindent \textit{Claim 2:}
the map $t  \mapsto K_t^* \rho_t$ is essentially separably valued, that is, there exists a measurable set $E \subset [0,1]$ such that $|E|=0$ and a countable set $S \subset D$ with the following property: for every $\e >0$ and $t \in [0,1] \smallsetminus E$ there exists $\f \in S$ such that $\norh{K_t^* \rho_t - i_t \f}<\e$.

\smallskip
\noindent \textit{Proof of Claim 2.} Let $T \subset [0,1]$ be a countable dense subset. Fix $t \in T$. By \ref{K3} the map $s \mapsto K_s^* \rho_t$ is strongly measurable and hence essentially separably valued by Theorem \ref{thm:pettis}. Therefore there exists a measurable set $E_t \subset [0,1]$ with $|E_t|=0$ and a countable subset $S_t \subset D$ with the following property: for every $\e >0$ and $s \in [0,1] \smallsetminus E_t$, there exists $\f \in S_t$ such that 
\begin{equation} \label{sep val3}
\nor{K_s^*\rho_t - i_s \f}_{H_s} < \e \,.	
\end{equation}
 Denote by $E:=\cup_{t \in T} E_t$. Since $T$ is countable, the set $E$ is measurable, and $|E|=0$. Moreover let 
 $S^0 := \cup_{t \in T} S_t$, so that $S^0 \subset D$ is countable. Define the set of averages of elements of $S^0$ as
 \[
 S:=\left\{ \frac{1}{k}\sum_{j=1}^k \f_j \, \colon \, k \geq 1, \, \f_j \in S^0 \right\} \,.
 \]
We have that $S \subset D$ is countable. Fix $\e >0$, $t \in [0,1] \smallsetminus E$. The claim follows by showing there exists $\f \in S$ such that
\begin{equation} \label{claim333}
\norh{K_t^* \rho_t - i_t \f} < \e	
\end{equation}
Indeed, by density, there exists a sequence $\{t_n\}$ in $T$ such that $t_n \to t$ as $n \to \infty$. Since $\rho_t \in \curves$ it follows that $\rho_{t_n} \weakstar \rho_t$ in $\M (\overline{\Om})$. By weak*-to-weak continuity of $K_t^*$ we have
$K_t^* \rho_{t_n} \weak K_t^* \rho_t$ weakly in $H_t$ 
as $n \to \infty$. By the Banach--Saks property in $H_t$ \cite[Ch VIII, Thm 1]{diestel2}, there exists a subsequence (not relabelled) such that
$
\frac{1}{n} \sum_{j=1}^n  K_t^* \rho_{t_j} \to K_t^* \rho_t$ strongly in $H_t$.
Hence we can choose $N \in \N$ such that 
\begin{equation} \label{st1}
\norh{ K_t^* \rho_t - \frac{1}{N} \sum_{j=1}^N  K_t^* \rho_{t_j}   } < \frac{\e}{2} \,.	
\end{equation}
Since $\{t_n\}$ is a sequence in $T$, by \eqref{sep val3} and the definitions of $S^0$ and $E$, we have that for every $n \in \N$ there exists $\f_n \in S^0$ such that 
\begin{equation} \label{st2}
\norh{K_t^* \rho_{t_n} - i_t \f_n} < \frac{\e}{2} \,.	
\end{equation}
 Define $\f :=\frac{1}{N} \sum_{j=1}^N \f_j$, so that $\f \in S$. By triangle inequality, linearity of $i_t$, and \eqref{st1}--\eqref{st2}, 
 \begin{align*}
 \norh{K_t^* \rho_t - i_t \f} & \leq \norh{K_t^* \rho_t - \frac{1}{N} \sum_{j=1}^N K_t^* \rho_{t_j}}     
 + \norh{ \frac{1}{N} \sum_{j=1}^N K_t^* \rho_{t_j} - i_t \f     }\\
 & \leq \frac{\e}{2}  + \frac{1}{N} \sum_{j=1}^N \norh{ K_t^* \rho_{t_j} - i_t \f_j  } < \e
 \end{align*}
which yields \eqref{claim333}.

\smallskip
\noindent \textbf{Part 2.}
Since $t \mapsto K_t^* \rho_t$ is strongly measurable, also $t \mapsto \norh{K_t^* \rho_t}$ is measurable. By \ref{K2} and Proposition \ref{prop:ub} we have
\[
\int_0^1  \norh{K_t^* \rho_t}^2 \, dt \leq C \int_0^1 \nor{\rho_t}^2_{\M(\overline{\Om})} \, dt \leq C \sup_t  \nor{\rho_t}^2_{\M(\overline{\Om})} = C {\left(\sup_t\nor{\rho_t}_{\M(\overline{\Om})} \right)}^2 < \infty \,.
\]
Hence by Theorem \ref{thm:int} we conclude that $K_t^*\rho_t$ is integrable and it belongs to $\Ltwo$.
\end{proof}

\begin{proof}[Proof of Proposition \ref{prop:well def}]
If $J(\rho,m,\mu) < \infty$ then also $B_\delta (\rho,m,\mu) < \infty$, hence by Proposition \ref{prop:bb prop} we have that $\rho \geq 0$, $m=v_t \rho$, $\mu = g_t \rho$ for Borel maps $v_t \colon X \to \R^d$, $g_t \colon X \to \R$ such that
\begin{equation} \label{b1}
\int_{X} (|v_t(x)|^2 + |g_t(x)|^2) \, d\rho(t,x) < \infty \,.
\end{equation}
By assumption $(\rho,m,\mu)$ solves the continuity equation, hence (Proposition \ref{prop:dis}) $\rho=dt \otimes \rho_t$ for some Borel family $\rho_t \in \M^+ (\overline{\Om})$. In  particular we have $m=dt \otimes (v_t \rho_t)$ and $\mu=dt \otimes (g_t \rho_t)$. By \eqref{b1} and Proposition \ref{prop:cont rep} we have that $\rho_t \in \pcurves$. Therefore $K_t^* \rho_t \in \Ltwo$ by Lemma \ref{prop:good definition}, and the first term in $J$ is finite.   
\end{proof}

\subsection{Existence of minimizers} \label{sec:existence}

The aim of this section is to prove that the functional $J$ defined in \eqref{def:J} admits at least one minimizer. Such minimizer is unique under suitable hypotheses on the operators $K_t^*$. The precise statement is the following. 

\begin{theorem} \label{thm:existence}
Let $f \in \Ltwo$ and $\alpha, \beta >0$, $\delta \in (0,\infty]$. Then there exists $(\rho^*,m^*,\mu^*) \in \mathcal{D}$ with $\rho^*=dt \otimes \rho_t^*$, $(t \mapsto \rho^*_t) \in \pcurves$, that solves the minimization problem 
\[
\inf_{(\rho,m,\mu) \in \M} J(\rho,m,\mu) \,.
\]	
If in addition $K_t^*$ is injective for a.e.~$t \in [0,1]$, then the minimizer is unique. 
\end{theorem}

The proof of the above theorem is based on the direct method of calculus of variations. Before proceeding to the proof, we will establish compactness and lower semicontinuity properties for the functional $J$. This is the object of the following two lemmas.

\begin{lemma}[Compactness for $J$] \label{lem:compactness}
	Let $f \in \Ltwo$ and $\alpha, \beta >0$, $\delta \in (0,\infty]$. Assume that there exists a constant $E \geq 0$ such that  the sequence $\{(\rho^n,m^n,\mu^n)\}$ in $\M$ satisfies
	\begin{equation} \label{hyp en bound}
	\sup_n J (\rho^n,m^n,\mu^n) \leq E \,.
	\end{equation}
 Then $\rho^n = dt \otimes \rho_t^n$ for some $(t \mapsto \rho_t^n) \in \pcurves$. Moreover there exists 
	$(\rho,m,\mu) \in \mathcal{D}$ with $\rho= dt \otimes \rho_t$, $(t \mapsto \rho_t) \in \pcurves$ such that, up to subsequences, 
	\begin{equation} \label{topology} \left\{
	\begin{gathered}
	(\rho^n,m^n,\mu^n) \weakstar (\rho,m,\mu) \,\, \text{ weakly* in } \,\, \M \,, \\
	\rho_t^n \weakstar \rho_t \,\, \text{ weakly* in }  \,\, \M(\olom)\,,\, \text{ for every } \,\, t \in [0,1]\,.   
	\end{gathered} \right.
	\end{equation}
\end{lemma}

\begin{proof}
	By the energy bound \eqref{hyp en bound}, there exists $\rho \in \M(X)$ such that, up to subsequences,
$\rho^n \weakstar \rho$ weakly* in $\M(X)$.
From \eqref{hyp en bound} we also have 
\begin{equation} \label{bound bb}
\sup_n B_\delta (\rho^n,m^n,\mu^n) \leq E \,,
\end{equation} 
so that Proposition \ref{prop:bb prop} implies $\rho^n \geq 0$, $m^n=v_t^n \rho^n$, $\mu^n = g_t^n \rho^n$ for Borel maps $v_t^n \colon X \to \R^d$, $g_t^n \colon X \to \R$ such that
\begin{equation} \label{b2}
\tilde{E}:= \sup_n \int_{X} (|v_t^n(x)|^2 + |g_t^n(x)|^2) \, d\rho^n(t,x) < \infty \,.
\end{equation}
By~\eqref{hyp en bound} we have $(\rho^n,m^n,\mu^n) \in \mathcal{D}$.  Hence Proposition \ref{prop:dis} implies that $\rho^n=dt \otimes \rho^n_t$ for some Borel family $\{\rho_t^n\}_{t \in [0,1]}$ in $\M^+ (\overline{\Om})$. In  particular we have $m^n=dt \otimes (v_t^n \rho_t^n)$ and $\mu=dt \otimes (g_t^n \rho_t^n)$. Hence by \eqref{b2} and Proposition \ref{prop:cont rep} we get that $(t \mapsto \rho_t^n) \in \pcurves$. Now notice that if $\rho \in \M^+(X)$, then by definition of $B_\delta$ (see \eqref{bb}) we infer
\begin{align*}
	\alpha B_\delta (\rho,m,\mu) + \beta \nor{\rho}_{\M (X)} & = \alpha \sup_{(a,b,c) \in C_0(X;K_\delta)} \int_{X}  a \, d\rho + \int_X b \cdot dm + \int_X c \, d\mu   + \int_{X} \beta \, d\rho \\
	& = \sup_{(a,b,c) \in C_0(X;\tilde{K}_\delta)} \int_{X}  a \, d\rho +\int_X b \cdot dm + \int_X c \, d\mu \,,
\end{align*}
where $K_\delta$ is defined in \eqref{keydelta} and 
\[
\tilde{K}_\delta := \left\{ (a,b,c) \in \R \times \R^d \times \R \, \colon \, a + \frac{1}{2} \left( |b|^2 + \frac{c^2}{\delta^2}\right) \leq \beta \right\} \,.
\]
Therefore, by taking $0<\e<\sqrt{2\beta}\min \{1,\delta\}$ we conclude
\begin{gather}
\e \nor{m}_{\M(X;\R^d)} = \sup_{\nor{b}_\infty \leq \e} \int_X b \cdot dm \leq  \alpha B_\delta (\rho,m,\mu) + \beta \nor{\rho}_{\M (X)} \leq J(\rho,m,\mu)\,, \nonumber \\
 \label{comp mu}
\e \nor{\mu}_{\M(X)} = \sup_{\nor{c}_{\infty} \leq \e} \int_X c\, d\mu  \leq \alpha B_\delta (\rho,m,\mu) + \beta \nor{\rho}_{\M (X)} \leq J(\rho,m,\mu)\,.
\end{gather}
By combining \eqref{hyp en bound} with the above estimates we conclude that (up to subsequences) $m^n \weakstar m$ and $\mu^n \weakstar \mu$ for some $m \in \M(X;\R^d), \mu \in \M(X)$. Since $\mathcal{D}$ is weak* closed, we get $(\rho,m,\mu) \in \mathcal{D}$. By Proposition \ref{prop:bb prop} the functional $B_\delta$ is weak* lower semicontinuous. Therefore \eqref{bound bb} implies $B_\delta (\rho,m,\mu)< \infty$ and by repeating the arguments above, we get that $\rho=dt \otimes \rho_t$, $m=dt \otimes (v_t \rho_t)$, $\mu = dt \otimes (g_t \rho_t)$ with $\rho_t \in \pcurves$ and
\[
\int_{X} (|v_t(x)|^2 + |g_t(x)|^2) \, d\rho(t,x) < \infty \,.
\]
In particular $(\rho,m,\mu) \in \mathcal{D}$. 
We will now show the second condition in \eqref{topology}.
Since $(\rho^n, m^n,\mu^n)$ solves the continuity equation, 
by Proposition \ref{prop:dis} we have that the map $t \mapsto \rho_t^n(\overline{\Om})$ belongs to $BV((0,1))$ with distributional derivative given by $\pi_\# \mu^n$, where we recall that $\pi \colon X \to (0,1)$ is the projection on the time coordinate. Therefore, by the embedding of $BV((0,1))$ into $L^\infty((0,1))$
\begin{equation} \label{bv est}
\begin{aligned}
\nor{\rho_t^n (\overline{\Om})}_{\infty} & \leq C
\nor{\rho_t^n (\overline{\Om})}_{\rm{BV}((0,1))} = C 
\int_0^1 \rho_t^n (\overline{\Om}) \, dt + C\nor{\de_t \rho_t^n(\overline{\Om})}_{\M((0,1))} \\
& = C \nor{\rho^n}_{\M(X)} + C \nor{\pi_\# \mu^n}_{\M((0,1))} 
\leq C \nor{\rho^n}_{\M(X)} + C \nor{ \mu^n}_{\M(X)} \leq C E\,,
\end{aligned}
\end{equation}
where we used \eqref{comp mu} and \eqref{hyp en bound}. 
Hence the set $\{\rho_t^n\}_{t,n}$ is uniformly bounded in $\M(\olom)$, so it belongs to some set $K \subset \M(\overline{\Om})$ which is weak* sequentially compact. Moreover as a consequence of Lemma \ref{lem:bound} we have that for every $t,s \in [0,1]$
\[
\nor{\rho_t^n - \rho_s^n}_{C^1(\olom)^*} \leq \sqrt{2C_n E_n} \, |t-s|^{1/2} \leq C  \, |t-s|^{1/2} \,,
\] 
where $E_n, C_n$ are the constants defined in the same lemma. The last inequality follows from \eqref{b2} and the fact that $\{\rho_t^n(\overline{\Om})\}_{t,n}$ is uniformly bounded, so that the constant $C>0$ does not depend on $n$. Hence by Proposition \ref{ascoli} there exists a subsequence (not relabelled) and a $C^1(\olom)^*$-continuous curve $\tilde{\rho}_t \colon [0,1] \to \M(\overline{\Om})$ such that
\begin{equation} \label{rho tilde}
\rho_t^n \weakstar \tilde{\rho}_t    \, \text{ weakly* in } \, \M(\overline{\Om})
\, \text{ for every } \, t \in [0,1] \,.
\end{equation}
In particular $\tilde{\rho}_t$ is narrowly continuous, since it is $C^1(\olom)^*$-continuous (this fact can be obtained by repeating the same argument given in the proof of Proposition \ref{prop:cont rep}). Notice that \eqref{rho tilde} implies that 
$\rho^n \weakstar \tilde{\rho}$ where $\tilde{\rho}:=dt \otimes \tilde{\rho}_t$. Hence $\tilde{\rho}=\rho$. By uniqueness of the disintegration we also get that $\rho_t = \tilde{\rho}_t$ and the thesis follows.
\end{proof}

\begin{lemma}[Lower semicontinuity for $J$] \label{lem:lower}
	Let $f \in \Ltwo$ and $\alpha, \beta >0$, $\delta \in (0,\infty]$. Assume that $(\rho^n,m^n,\mu^n) \in \mathcal{D}$ with $\rho^n=dt \otimes \rho^n_t$, $(t \mapsto \rho_t^n)  \in \pcurves$ is such that $(\rho^n,m^n,\mu^n)$ converges to $(\rho,m,\mu)$ in the sense of \eqref{topology}, where $\rho=dt \otimes \rho_t$, $(t \mapsto \rho_t)  \in \pcurves$.
	Then we have 
	\begin{equation} \label{weak cont}
	K_t^*\rho_t^n \weak K_t^*\rho_t \,\, \text{ weakly in } \,\, \Ltwo \,. 
	\end{equation}
	In particular $J$ is lower semicontinuous with respect to the convergence in \eqref{topology}, that is,
	\begin{equation} \label{lower thesis}
	J(\rho,m,\mu) \leq \liminf_n J(\rho^n,m^n,\mu^n) \,.
	\end{equation}
\end{lemma}

\begin{proof} Let us start by showing \eqref{weak cont}. To this end, fix $g \in \Ltwo$. By assumption we have that $\rho_t^n \weakstar \rho_t$ weakly* in $\M(\olom)$, for every $t \in [0,1]$. In particular, by \ref{K1bis}, we have
$K_t^* \rho_t^n \weak K_t^* \rho_t$ weakly in $H_t$ for a.e. $t \in [0,1]$,
so that
\begin{equation} \label{point leb}
 \inner{K_t^* \rho_t^n}{g(t)} \to  \inner{K_t^* \rho_t}{g(t)} \, \text{ for a.e. } \,  t \in [0,1] 
\end{equation}
as $n \to \infty$. By proceeding as in \eqref{bv est} we obtain
\[
\nor{\rho_t^n(\olom)}_{\infty} \leq C \nor{\rho^n}_{\M(X)} + C \nor{\mu^n}_{\M(X)} \leq C \,,
\]
for some constant $C \geq 0$, since $\rho^n$ and $\mu^n$ are uniformly bounded by weak* convergence in $\M(X)$. Hence by Cauchy--Schwarz and by \ref{K2} we have
\[
\left| \inner{K_t^* \rho_t^n}{g(t)} \right| \leq \norh{K_t^* \rho_t^n} \, \norh{g(t)} \leq C \rho_t^n(\olom) \, \norh{g(t)} \leq C  \norh{g(t)} \,. 
\]
Since $g \in \Ltwo$ we have that $t \mapsto  \norh{g(t)}$ belongs to $L^1((0,1))$. By combining the above estimate with \eqref{point leb} and invoking the classic dominated convergence theorem we conclude \eqref{weak cont}. 

Let us now prove the remaining part of the Lemma. From \eqref{weak cont} we have that
$K_t^* \rho_t^n - f_t$ converges weakly to $K_t^* \rho_t - f_t$  in $\Ltwo$,
therefore by lower semicontinuity of the norm with respect to the weak convergence we have
\[
\frac{1}{2} \nor{K_t^* \rho_t - f_t}_{L^2}^2 \leq 
\liminf_n \frac{1}{2} \nor{K_t^* \rho_t^n - f_t}_{L^2}^2 \,.
\]
Moreover $B_\delta$ is weak* lower semicontinuous by Proposition \ref{prop:bb prop}, thus 
\begin{align*}
J(\rho,m,\mu) & \leq \liminf_n \frac{1}{2} \nor{K_t^* \rho_t^n - f_t }_{L^2}^2   + \liminf_n \left( \alpha B_\delta (\rho^n,m^n,\mu^n) + \beta \nor{\rho^n}_{\M(X)}  \right)\\ 
& \leq \liminf_n J(\rho^n,m^n ,\mu^n) \,,
\end{align*}
and \eqref{lower thesis} follows. 
\end{proof}

\begin{proof}[Proof of Theorem \ref{thm:existence}]
\smallskip
\noindent \textit{Existence:}
Set $\rho:=dt \otimes \sigma$ with $\sigma \in \M^+(\overline{\Om})$. Then $(\rho,0,0) \in \mathcal{D}$, and $B_\delta(\rho,0,0)=0$ by (iv) in Proposition \ref{prop:bb prop}. Moreover $(t \mapsto K_t^* \sigma) \in \Ltwo$ by Lemma \ref{prop:good definition}, so that $J(\rho,0,0)< \infty$ and the infimum is finite. Let $(\rho^n,m^n,\mu^n)$ be a minimizing sequence, that is,
$J(\rho^n,m^n,\mu^n) \to \inf_{\M} J$
as $n \to \infty$. Therefore $\sup_n J(\rho^n,m^n,\mu^n) \leq E$
for some constant $E \geq 0$. From Lemma \ref{lem:compactness} we have $\rho^n = dt \otimes \rho_t^n$ and $\rho_t^n \in \pcurves$. Moreover there exists 
	$(\rho^*,m^*,\mu^*) \in \mathcal{D}$ with $\rho^*= dt \otimes \rho_t^*$, $(t \mapsto \rho_t^*) \in \pcurves$ such that, up to subsequences, $(\rho^n,m^n,\mu^n)$ converges to $(\rho^*,m^*,\mu^*)$ in the sense of \eqref{topology}. By Lemma \ref{lem:lower} and the fact that $(\rho^n,m^n,\mu^n)$ is a minimizing sequence we conclude that $(\rho^*,m^*,\mu^*)$ is a minimizer for $J$.

\smallskip

\noindent \textit{Uniqueness:} Assume that $K_t^*$ is injective for a.e.~$t \in [0,1]$. The term $B_\delta$ is convex by Proposition \ref{prop:bb prop}. Also the term $\nor{\cdot}_{\M (X)}$ is convex as it is a norm. Since minimizers are necessarily of the form $(\rho,m,\mu) \in \mathcal{D}$ with $\rho \geq 0$,  in order to prove uniqueness it is sufficient to show that 
\begin{equation} \label{conv map}
\rho_t \mapsto \nor{K_t^* \rho_t - f_t}^2_{L^2}
\end{equation}
is strictly convex for $\rho_t \in C_{\rm w}([0,1];\M^+(\overline{\Om}))$. First, consider $\rho_t  \in C_{\rm w}([0,1];\M^+(\overline{\Om}))$ such that $\rho_t \not\equiv 0$. As a consequence the set 
$E:=\{t \in [0,1] \, \colon \, \rho_t(\overline{\Om}) \neq 0 \}$
is open (by continuity of $t \mapsto \rho_t(\olom)$) and not empty, so $|E|>0$. Let
$F:= \left\{ t \in [0,1] \, \colon \, K_t^* \, \text{ is injective}\right\}$.
By assumption we have $|[0,1] \smallsetminus F|=0$. Therefore
\begin{equation} \label{non zero}
\nor{K_t^* \rho_t }^2_{L^2} = \int_0^1 \norh{K_t^* \rho_t}^2 \, dt \geq
\int_{E \cap F} \norh{K_t^* \rho_t}^2 \, dt > 0
\end{equation}
since $K_t^* \rho_t \neq 0$ for $t \in E \cap F$.  
Now let $\rho_t^1, \rho_t^2 \in C_{\rm w}([0,1];\M^+(\overline{\Om}))$ with $\rho^1_t \not\equiv \rho^2_t$ and $\lambda \in (0,1)$. The coefficient of the leading term of 
$
\lambda \mapsto \nor{ K_t^*( \lambda \rho^1_t +(1-\lambda)\rho^2_t ) }_{L^2}^2$ 
is $\nor{K_t^*(\rho_t^1-\rho_t^2) }_{L^2}^2$, 
which is non zero by \eqref{non zero}, since $\rho_t^1 \not\equiv \rho_t^2$. Hence the map in \eqref{conv map} is strictly convex and we conclude.
\end{proof}

\subsection{Regularization properties} \label{sec:stability}

In this section we denote by $f^\dagger \in \Ltwo$ the exact data and by $f^\gamma \in \Ltwo$ the noisy data for the noise level $\gamma>0$, that is, $\nor{f^\gamma - f^\dagger}_{L^2} \leq \gamma$. For a datum $f \in \Ltwo$ and parameters $\alpha,\beta>0$ we adopt the following notation: 
\[
J_{\alpha,\beta,f}(\rho,m,\mu):= \frac{1}{2} \nor{K^*_t \rho_t - f_t}_{L^2}^2 + \alpha B_\delta(\rho,m,\mu) + \beta \nor{\rho}_{\M(X)} \,
\]
if $(\rho,m,\mu) \in \mathcal{D}$ and $J_{\alpha,\beta,f}(\rho,m,\mu)= \infty$ otherwise.

\begin{theorem}[Stability] \label{thm:stability}
Assume that $f^n \to f^\gamma$ strongly in $\Ltwo$ and that 
\begin{equation} \label{argmin}
(\rho^n,m^n,\mu^n) \in \argmin_{(\rho,m,\mu) \in \mathcal{\M}} J_{\alpha,\beta,f^n} (\rho,m,\mu) \,.
\end{equation}
Then $(\rho^n,m^n,\mu^n) \in \mathcal{D}$ with $\rho^n=dt \otimes \rho_t^n$, $\rho_t^n \in \pcurves$. Moreover $(\rho^n,m^n,\mu^n)$ admits a subsequence converging in the sense of \eqref{topology}. The limit of each converging subsequence of $(\rho^n,m^n,\mu^n)$ is a minimizer of $J_{\alpha,\beta,f^\gamma}$. 	
\end{theorem}
\begin{proof}
A sequence $(\rho^n,m^n,\mu^n) \in \mathcal{D}$ satisfying \eqref{argmin} exists by Theorem \ref{thm:existence}. By the same theorem it also follows that $\rho^n=dt \otimes \rho_t^n$ with $\rho_t^n \in \pcurves$. 
We have
\[
J_{\alpha,\beta,f^\gamma} (\rho^n,m^n,\mu^n) \leq   
2 J_{\alpha,\beta,f^n} (\rho^n,m^n,\mu^n) +  \nor{f^\gamma_t - f^n_t}_{L^2}^2
\,.
\]
Since $(\rho^n,m^n,\mu^n)$ is a minimizer for $J_{\alpha,\beta,f^n}$ we can test \eqref{argmin} against $(0,0,0) \in \mathcal{D}$ to obtain
\[
\sup_n J_{\alpha,\beta,f^\gamma} (\rho^n,m^n,\mu^n) \leq  \nor{f_t^n}_{L^2}^2 + \nor{f_t^n - f^\gamma_t}^2_{L^2} < \infty \,,
\]
where last inequality follows from the convergence $f^n \to f^\gamma$ in $\Ltwo$. 
By applying Lemma \ref{lem:compactness} to $J_{\alpha,\beta,f^\gamma}$, there exists $(\tilde{\rho},\tilde{m}, \tilde{\mu}) \in \mathcal{D}$, with $\tilde{\rho}=dt \otimes \tilde{\rho}_t$, $\tilde{\rho}_t \in \pcurves$ such that
 $(\rho^n,m^n,\mu^n)$ converges to $(\tilde{\rho},\tilde{m}, \tilde{\mu})$ in the sense of \eqref{topology}. 
 We are left to show that $(\tilde{\rho},\tilde{m}, \tilde{\mu})$ is a minimizer for $J_{\alpha,\beta,f^\gamma}$. 
 Since $\rho_t^n \weakstar \tilde{\rho}_t$ for every $t \in [0,1]$, by Lemma \ref{lem:lower} and the convergence $f^n \to f^\gamma$ we have $(K_t^* \rho_t^n - f^n_t) \weak ( K_t^* \tilde{\rho}_t - f^\gamma_t)$ weakly in $\Ltwo$. Also by Lemma \ref{lem:lower},
\[
 J_{\alpha,\beta,f^\gamma} (\tilde{\rho},\tilde{m}, \tilde{\mu})  \leq
 \liminf_n J_{\alpha,\beta,f^n} (\rho^n,m^n,\mu^n)
 \leq \liminf_n J_{\alpha,\beta,f^n} (\rho,m,\mu) =   
 J_{\alpha,\beta,f^\gamma} (\rho,m,\mu)  
\]
 for every $(\rho,m,\mu) \in \mathcal{M}$, since \eqref{argmin} holds and $f^n \to f^\gamma$. Hence $(\tilde{\rho},\tilde{m}, \tilde{\mu})$ is a minimizer.
\end{proof}

We are now interested in studying properties of the minimizers of $J_{\alpha,\beta,f^\gamma}$ for vanishing noise level, that is, for data such that $\nor{f^\gamma - f^\dagger}_{L^2} \leq \gamma$ for every $\gamma \geq 0$. To this end, we need to understand how the regularization term
\[
\alpha B_\delta(\rho,m,\mu) + \beta \nor{\rho}_{\M(X)}
\]
behaves for fixed argument $(\rho,m,\mu) \in \mathcal{D}$. Since multiple parameters are involved, we will also allow $\alpha$ and $\beta$ to take the value $\infty$. We define 
\begin{equation} \label{beta inf}
\beta \nor{\cdot}_{\M(X)} := I_{\{0\}}  \,\, \text{ if } \,\, \beta = \infty,
\end{equation}
where $I_{\{0\}}$ denotes the convex indicator function of the set ${\{0\}}$. 
In order to give a similar definition for the case $\alpha=\infty$ we first need to characterize the subset of $\mathcal{D}$ where $B_\delta(\rho,m,\mu)=0$.

\begin{proposition} \label{prop:zero transport}
	Assume that $(\rho,m,\mu) \in \mathcal{D}$. We have that $B_\delta (\rho,m,\mu) = 0$ if and only if $m=0$, $\mu=0$ and $\rho = dt \otimes \sigma$ for some $\sigma \in \M^+(\olom)$. 
\end{proposition}

\begin{proof}
By Proposition \ref{prop:bb prop} point (iv) we have that $B_\delta(dt \otimes \sigma, 0,0)=0$ for any $\sigma \in \M^+(\olom)$. Conversely, assume that $(\rho,m,\mu) \in \mathcal{D}$ is such that $B_\delta (\rho,m,\mu) = 0$. In particular the energy is finite, so points (iii)--(iv) of Proposition \ref{prop:bb prop} imply that $\rho \geq 0$, $m=v \rho$, $\mu = g \rho$ for some Borel maps $v \colon X \to \R^d$, $g \colon X \to \R$, and we have
$
B_\delta (\rho,m,\mu) = \frac{1}{2} \int_X (|v|^2 + \delta^2 g^2) \, d\rho$. 	
Since $B_\delta (\rho,m,\mu)=0$ and $\rho \geq 0$, we conclude that $m=0$ and $\mu=0$. By assumption $(\rho,0,0)$ solves the continuity equation in the sense of \eqref{cont weak}, therefore Proposition \ref{prop:dis} guarantees that $\rho = dt \otimes \rho_t$ for some Borel family $\rho_t \in \M^+(\olom)$. 
Since $v=0$ and $g=0$ a.e.~in $X$, we can apply Proposition \ref{prop:cont rep} and conclude that $(t \mapsto \rho_t) \in \pcurves$. In particular, for every $0 \leq t_1 \leq t_2 \leq 1$ and $\f(t,x):=a(x)$ with $a \in C^1(\olom)$, formula \eqref{cont boundary} reads
$
\int_{\olom} a ( x) \, d\rho_{t_1} (x) =
\int_{\olom} a( x) \, d\rho_{t_2} (x)
$.
By a density argument one can show that the previous holds for all $a \in C(\olom)$, and hence $\rho_t=\rho_0$ for every $t \in [0,1]$. The thesis follows by setting $\sigma:=\rho_0$.
\end{proof}

Proposition \ref{prop:zero transport} motivates the following definition:
\begin{equation}\label{alpha inf}
\alpha B_\delta := I_Z  \,\, \text{ if } \,\, \alpha = \infty,
\end{equation}
where 
$
Z:=\left\{ (\rho,0,0) \in \mathcal{D} \, \colon \, \rho = dt \otimes \sigma\,, \,\, \sigma \in \M^+(\olom) \right\} \,.
$
We are now in the position to define minimal energy solutions of the inverse problem
\begin{equation} \label{inverse dagger 2}
K_t^* \rho_t = f_t^\dagger \, \text{ for a.e. } \, t \in [0,1] \,.
\end{equation}

\begin{definition}[Minimal energy solution]
	Let $f^\dagger \in \Ltwo$ and $\alpha^*,\beta^* \in [1, \infty]$, $\delta \in (0,\infty]$. We say that $(\rho^\dagger,m^\dagger,\mu^\dagger) \in \M$ is a \textit{minimal energy solution} of  \eqref{inverse dagger 2} with parameters $\alpha^*, \beta^*$ if
\[
	(\rho^\dagger,m^\dagger,\mu^\dagger) \in \argmin_{(\rho,m,\mu) \in \mathcal{D}} \left\{\alpha^* B_\delta (\rho,m,\mu) + \beta^* \nor{\rho}_{\M(X)} \, \colon \, K_t^* \rho_t = f_t^\dagger \,\, \text{ a.e.~in } \, [0,1] \right\} \,
\]
\end{definition}

In the following theorem we show that the minimizers for vanishing noise level converge in the sense of \eqref{topology} to an energy minimizing solution of the inverse problem \eqref{inverse dagger 2}.

\begin{theorem}[Convergence for vanishing noise level] \label{thm:vanishing}
Let $f^\dagger \in \Ltwo$ be the exact data and $\{f^n\}$ be a sequence of noisy data such that $\nor{f^n - f^\dagger}_{L^2} \leq \gamma_n$. Let $\alpha_n, \beta_n >0$ be such that
\begin{equation} \label{coeff}
\alpha_n, \beta_n \searrow 0 \quad \text{ and } \quad \frac{\gamma_n^2}{\alpha_n}, \frac{\gamma_n^2}{\beta_n} \to 0 \quad \text{ as } \quad n \to \infty \,.
\end{equation}
Let $c_n :=\min \{\alpha_n,\beta_n\}$, $\tilde{\alpha}_n:=\alpha_n/c_n$, $\tilde{\beta}_n:=\beta_n/c_n$ so that, up to subsequences,  $\tilde{\alpha}_n \to \alpha^*$ and $\tilde{\beta}_n \to \beta^*$ as $n \to \infty$, with $\alpha^*,\beta^* \in [1,\infty]$. 
Assume there exists $(\rho^\dagger,m^\dagger,\mu^\dagger) \in \mathcal{D}$ satisfying $\rho^\dagger =dt \otimes \rho^\dagger_t$, $\rho^\dagger_t \in \pcurves$, \eqref{inverse dagger 2} and
\begin{equation}
\alpha^* B_\delta(\rho^\dagger,m^\dagger,\mu^\dagger) + \beta^* \nor{\rho^\dagger}_{\M(X)} < \infty \,. \label{finite min norm} 
\end{equation}
Let $(\rho^n,m^n,\mu^n) \in \M$ be such that
\begin{equation} \label{argminw}
(\rho^n,m^n,\mu^n) \in \argmin_{(\rho,m,\mu) \in \mathcal{M}} J_{\alpha_n,\beta_n,f^n}(\rho,m,\mu) \,.
\end{equation}
Then $\rho^n=dt \otimes \rho_t^n$ with $\rho_t^n \in \pcurves$, and $(\rho^n,m^n,\mu^n)$ converges to $(\rho^*,m^*,\mu^*)$ in the sense of \eqref{topology}, up to subsequences. Moreover, every such weak* limit of $(\rho^n,m^n,\mu^n)$ is a minimal energy solution of \eqref{inverse dagger 2} with parameters $\alpha^*$ and $\beta^*$.
 \end{theorem}

\begin{proof}
First notice that $c_n \to 0$ and $\gamma^2_n/c_n \to 0$ as $n \to \infty$ by \eqref{coeff}. If $\tilde{\alpha}_n \to \infty$ or $\tilde{\beta}_n \to \infty$, we set $\alpha^*:=\infty$ and $\beta^*:=\infty$ respectively. If either of the sequences do not diverge to $\infty$, it is possible to find accumulation points $\alpha^*, \beta^* \in [1,\infty)$. In particular, up to extracting a subsequence, we can assume that $\tilde{\alpha}_n \to \alpha^*$ and $\tilde{\beta}_n \to \beta^*$ as $n \to \infty$.  A sequence $(\rho^n,m^n,\mu^n)$ satisfying \eqref{argminw} exists by Theorem \ref{thm:existence}. By the same theorem it also follows that $\rho^n=dt \otimes \rho_t^n$ with $\rho_t^n \in \pcurves$. 
By testing \eqref{argminw} against $(\rho^\dagger,m^\dagger,\mu^\dagger)$ and using \eqref{inverse dagger 2}--\eqref{coeff} we get
\begin{equation} \label{stab000}
J_{\alpha_n,\beta_n,f^n}(\rho^n,m^n,\mu^n) \leq 
\frac{\gamma_n^2}{2} + \alpha_n B_\delta (\rho^\dagger,m^\dagger,\mu^\dagger) + \beta_n \nor{\rho^\dagger}_{\M(X)}  \to 0\,.
\end{equation}
In particular $(K_t^* \rho_t^n - f_t^n) \to 0$ in $\Ltwo$. Since by assumption $f^n \to f^\dagger$, we obtain $K_t^* \rho^n_t \to f_t^\dagger$ in $\Ltwo$. Dividing the inequality at \eqref{stab000} by $c_n$, taking the limes superior and keeping \eqref{coeff} intro account yields
\begin{equation} \label{stab1}
\limsup_n \left(\tilde{\alpha}_n B_\delta (\rho^n,m^n,\mu^n) +  \tilde{\beta}_n \nor{\rho^n}_{\M(X)}  \right)\leq
 \alpha^* B_\delta (\rho^\dagger,m^\dagger,\mu^\dagger) + \beta^* \nor{\rho^\dagger}_{\M(X)} \,.
\end{equation}
Notice that the right hand side in \eqref{stab1} is always bounded, thanks to definitions \eqref{beta inf}, \eqref{alpha inf} and assumption \eqref{finite min norm}.
By definition we have $\tilde{\alpha}_n, \tilde{\beta}_n \geq 1$, thus
\[
\limsup_n  J_{1,1,f^\dagger} (\rho^n,m^n,\mu^n)  \leq %
 \limsup_n \left( \tilde{\alpha}_n B_\delta (\rho^n,m^n,\mu^n) +  \tilde{\beta}_n \nor{\rho^n}_{\M(X)}  \right) < \infty 
\]
where we used \eqref{stab1} and the convergence $K_t^* \rho^n_t \to f_t^\dagger$.  
Therefore an application of Lemma \ref{lem:compactness} guarantees the existence of $(\rho^*,m^*,\mu^*) \in \mathcal{D}$ with $\rho^*=dt \otimes \rho^*_t $, $\rho^*_t \in \pcurves$, such that, up to subsequences, $(\rho^n,m^n,\mu^n)$ converges to $(\rho^*,m^*,\mu^*)$ in the sense of \eqref{topology}. In particular by Lemma \ref{lem:lower} we have $K_t^* \rho^n_t \weak K_t^* \rho^*_t$ weakly in $\Ltwo$. Since we already proved that $K_t^* \rho^n_t \to f^\dagger_t$, by uniqueness of the weak limit we have 
\begin{equation} \label{stab3}
K_t^* \rho^*_t = f_t^\dagger  \,\, \text{ a.e.~in } \,\, [0,1] \,.
\end{equation}
We are left to show that $(\rho^*,m^*,\mu^*)$ is an energy minimizing solution of \eqref{inverse dagger 2}. By Lemma \ref{lem:lower}%
\begin{equation} \label{stab2}
\begin{aligned}
& \alpha^*  B_\delta  (\rho^*,m^*,\mu^*)  + \beta^* \nor{\rho^*}_{\M(X)}  \leq 
\liminf_n   \left( \tilde{\alpha}_n B_\delta (\rho^n,m^n,\mu^n) + \tilde{\beta}_n \nor{\rho^n}_{\M(X)} \right) \\
& \leq \limsup_n \left(  \tilde{\alpha}_n B_\delta (\rho^n,m^n,\mu^n) + \tilde{\beta}_n \nor{\rho^n}_{\M(X)} \right) 
 \leq  \alpha^* B_\delta (\rho^\dagger,m^\dagger,\mu^\dagger) + \beta^* \nor{\rho^\dagger}_{\M(X)} 
\end{aligned}
\end{equation}
where we used \eqref{stab1} and that $\tilde{\alpha}_n \to \alpha^*$, $\tilde{\beta}_n \to \beta^*$. Replacing $(\rho^\dagger,m^\dagger,\mu^\dagger)$ by an arbitrary solution of \eqref{inverse dagger 2} with finite energy $B_\delta$, the argument can be repeated, and from \eqref{stab3}--\eqref{stab2} we conclude that $(\rho^*,m^*,\mu^*)$ is an energy minimizing solution of \eqref{inverse dagger 2}.
\end{proof}

\section{Application to dynamic undersampled MRI} \label{sec:MRI}

We will now detail on the application of the above results to dynamic
magnetic resonance imaging as outlined in the introduction.
Let $\Om \subset \R^2$ be an open bounded domain representing the image frame, and $c_j \in C_0(\R^2;\C)$ for $j=1,\dots,N$ with $N \geq 1$ be the coil sensitivities. Let $\sigma_t \in \M^+ (\R^2)$ for $t \in [0,1]$ be a family of measures such that
\begin{enumerate}[label=\textnormal{(M\arabic*)}]
\item $\nor{\sigma_t}_{\M(\R^2)} \leq C$ for a.e.~$t \in [0,1]$, where $C >0$ does not depend on $t$, \label{M1}
\item the map $t \mapsto \int_{\R^2} \f(x) \, d\sigma_t(x)$ is measurable for each $\f \in C_0(\R^2;\C)$. 	\label{M2}
\end{enumerate}
Let $D:=C_0 (\R^2;\C^N)$ be the Banach space normed by $\nor{\f}_{\infty}:=\max_{\{j=1,\dots,N\}} \max_{x \in \R^2}|\f^j(x)|$, where $\f = (\f^1,\dots,\f^N)$. Define Hilbert spaces $H_t:=L^2_{\sigma_t}(\R^2;\C^N)$, equipped with the norm $\norh{h}^2:=\sum_{j=1}^N \int_{\R^2} |h^j(x)|^2 \, d\sigma_t(x)$, where  $h=(h^1,\dots,h^N)$. Define $i_t \colon D \to H_t$ as the identity map, acting component-wise. Note that here we are interpreting $D$ and $H_t$ as real vector spaces.  
For a measure $\rho \in \M(\olom;\C)$ we denote its Fourier transform as 
\[
\F \rho (x):= \frac{1}{2\pi} \int_{\R^2} e^{-i \omega \cdot x } \, d\rho(\omega) \,,
\]
so that $\F \rho \in C(\R^2;\C)$. Notice that in the above definition we extend $\rho$ to be zero outside of $\olom$. 
For each $t \in [0,1]$, define the linear operator $K_t^* \colon \M(\olom) \to H_t$ as 
\begin{equation} \label{K MRI}
K_t^* \rho := (\F(c_1\rho), \dots, \F(c_N\rho)) \,.
\end{equation}
In the MRI context, the family of measures $\rho_t \in \M^+(\olom)$ for $t \in [0,1]$ represents the proton density at each time step. Given some data $f \in \Ltwo$, we want to reconstruct a solution to the dynamic inverse problem
\begin{equation} \label{dyn MRI inv}
K_t^* \rho_t = f_t \,\, \text{ for a.e. } \,\, t \in [0,1] \,.
\end{equation}
As proposed in the previous sections, we relax the problem to measures $\rho \in \M(X)$, with $X:=(0,1) \times \olom$, and minimize the functional $J$ introduced in \eqref{def:J}. Under the assumptions \ref{M1}--\ref{M2} the functional $J$ admits at least one minimizer, and minimizers are unique under suitable additional assumptions. This claim is the object of the following theorem.

\begin{theorem} \label{thm:MRI}
Let $\alpha,\beta >0$, $\delta \in (0,\infty]$, $f \in \Ltwo$. Let $\{\sigma_t\}$ for $t \in [0,1]$ be a family of Radon measures in $\M^+ (\R^2)$ satisfying \ref{M1}--\ref{M2}. Let $c_1,\dots,c_N \in C_0(\R^2;\C)$ be coil sensitivities. Then the regularization of \eqref{dyn MRI inv} according to
\[
\inf_{(\rho,m,\mu) \in \mathcal{D}} \frac{1}{2} \sum_{j=1}^N \int_0^1 \nor{\F(c_j \rho_t)-f_t^j}_{L^2_{\sigma_t}(\R^2;\C)}^2   \, dt  + 
\alpha B_\delta (\rho,m,\mu) + \beta \nor{\rho}_{\M(X)}
\]	
admits a solution $(\rho^*,m^*,\mu^*) \in \mathcal{D}$ with $\rho^*=dt \otimes \rho_t^*$ and the curve $t \mapsto \rho_t^*$ belonging to $\pcurves$. If in addition the supports of the measures $\sigma_t$ have non empty interior for a.e.~$t \in [0,1]$, and the vector of coil sensitivities satisfies $c(x) \neq 0$ for every $x \in \olom$, 
then the minimizer is unique.    
\end{theorem}

Before proceeding with the proof, we want to show how this analytical  framework allows us to treat a wide variety of sampling patterns. We will give two examples.

\begin{example}[Continuous sampling] Let $\Om:=(-\frac12,\frac12)^2$ %
and for $t \in [0,1]$ define the line $L_t:=(-\frac12,\frac12) \times \{ t-\frac12\}$. Set $\sigma_t:=\mathcal{H}^1 \zak L_t$, that is, the restriction of the $1$-dimensional Hausdorff measure to the lines $L_t$. It is immediate to check that $\sigma_t \in \M^+(\R^2)$ satisfies \ref{M1}--\ref{M2}: indeed $\nor{\sigma_t}_{\M^+(\R^2)} = \mathcal{H}^1(L_t)= 1$, while the map $t \mapsto \int_{\R^2} \f(x) d\sigma_t(x)$ is continuous for $\f$ in $C_0(\R^2;\C)$. In the same way we can treat radial sampling, by setting $L_t$ to be a collection of diameters through the origin, evolving in time (see Example~\ref{ex2}).%
\end{example}

\begin{example}[Compressed-sensing sampling]
In this example we propose to sample along a finite collection of moving points in an open bounded domain $\Om \subset\R^2$. To be more specific, fix $M \in \N$, $M \geq 1$ and for every $j =1,\dots,M$ let $t \in [0,1] \mapsto x_t^j \in \Om$ be a measurable curve. %
For a.e.~$t \in [0,1]$ define $P_t:=\{ x_t^1,\dots,x_t^M\}$ and $\sigma_t:=\mathcal{H}^0 \zak P_t = \sum_{j=1}^M \delta_{x_t^j}$. Notice that \ref{M1} is satisfied since $\nor{\sigma_t}_{\M(\R^2)} =M$. Given a map $\f \in C_0(\R^2;\C)$ we have that $t \mapsto \int_{\R^2} \f (x) \, d\sigma_t(x) = \sum_{j=1}^M \f (x_t^j)$   
is measurable by construction. Therefore also $\ref{M2}$ is satisfied. 
\end{example}

We now want to prove Theorem \ref{thm:MRI}. Before that,
we need a preliminary lemma, stating that under \ref{M1}--\ref{M2} the above definitions of $D, H_t, i_t, K_t^*$ satisfy assumptions \ref{H1}--\ref{H3} and \ref{K1}--\ref{K3}.

\begin{lemma} \label{MRI H}
Assume that \ref{M1}--\ref{M2} hold. The spaces $D$, $H_t$ and the operators $i_t$ satisfy \ref{H1}--\ref{H3}.	
Moreover the operators $K_t^*$ in \eqref{K MRI} satisfy \ref{K1}--\ref{K3}.
\end{lemma}

\begin{proof}
Notice that $i_t$ is linear and continuous, with $\nor{i_t}^2 \leq N\nor{\sigma_t}_{\M(\R^2)}$. In particular \ref{H1} follows from \ref{M1}. Moreover \ref{H2} is trivially satisfied. Finally for $\f,\psi \in D$ we have
\[
t \mapsto \inner{i_t \f}{i_t \psi} = \Re \left( \sum_{j=1}^N \int_{\R^2} \f^j (x) \overline{\psi^j(x)} \, d\sigma_t (x) \right) \,,
\]  
which is measurable by \ref{M2}, as it is the real part of a sum of measurable maps. Hence \ref{H3} is also satisfied. 
Let us now show that \ref{K1}--\ref{K3} hold. For $\rho \in \M(\olom)$ we have
\begin{equation} \label{bound fourier}
| \F (c_j \rho)(x)  | = \frac{1}{2\pi} \left|   \int_{\R^2} e^{-i \omega \cdot x} c_j(\omega) \, d\rho(\omega) \right| \leq 
\frac{\nor{c_j}_{\infty}}{2\pi} \, \nor{\rho}_{\M(\olom)} \,.    
\end{equation}
Hence, each $ \F (c_j \rho)$ is square integrable with respect to $\sigma_t$, so that $K_t^*$ maps $\M(\olom)$ into $H_t$. Moreover, by the above estimate we also have
\[
\norh{K_t^* \rho}^2 = \sum_{j=1}^N \int_{\R^2} | \F (c_j \rho)(x)|^2 \, d\sigma_t (x) \leq \frac{N}{4 \pi^2} \nor{c}^2_{\infty} \nor{\rho}_{\M(\olom)}^2 \nor{\sigma_t}_{\M(\R^2)} \,,
\]
where $c:=(c_1,\dots,c_N)$ is the vector of coil sensitivities. Therefore $K^*_t$ is continuous, with
$\nor{K^*_t}^2 \leq\frac{N}{4 \pi^2} \nor{c}_{\infty}^2 \nor{\sigma_t}_{\M(\R^2)}$
and \ref{K2} is satisfied because of assumption \ref{M1}. Let us show that $K_t^*$ is weak*-to-weak continuous. To this end, let $\{\rho_n\}$ in $\M(\olom)$ be such that $\rho_n \weakstar \rho$. Since $\rho_n,\rho$ are supported in the compact set $\olom$, it follows that 
\begin{equation} \label{dct 1}
\F(c_j\rho_n) \to \F(c_j \rho) \,\, \text{ pointwise in } \,\, \R^2 \,,
\end{equation}
as $n \to \infty$. 
Moreover, by weak* convergence we have $\sup_n \nor{\rho_n}_{\M(\olom)}< \infty$. As a consequence of \eqref{bound fourier}, there exists a constant $C \geq 0$ such that
 \begin{equation} \label{dct 2}
 \sup_n \nor{\F (c_j \rho_n)}_{\infty} \leq C  \,\, \text{ for every } \,\, j=1, \dots, N \,.
 \end{equation}
By invoking the dominated convergence theorem in conjunction with \eqref{dct 1}--\eqref{dct 2} we conclude that $K_t^* \rho_n \weak K_t^* \rho$ weakly in $L^2_{\sigma_t}(\R^2;\C^N)$. Hence \ref{K1bis} is satisfied and, as a consequence, $K_t^*$ is the adjoint of some linear continuous operator $K_t \colon H_t \to C(\olom)$. Finally we need to show \ref{K3}: that the map $t \mapsto K_t^* \rho$ is strongly measurable according to Definition \ref{def:strong}, for every fixed $\rho \in \M(\olom)$. Notice that the space $D=C_0(\R^2;\C^N)$ is separable, hence by Proposition \ref{rem:sep} it is sufficient to show that $t \mapsto K_t^* \rho$ is weakly measurable according to Definition \ref{def:strong}. %
However, since $\F (c_j\rho) \in C(\R^2;\C)$ is bounded (see \eqref{bound fourier}) and maps in $D$ are continuous, %
this is an immediate consequence of \ref{M2}. 
\end{proof}

\begin{proof}[Proof of Theorem \ref{thm:MRI}]
The existence of a minimizer follows from Proposition \ref{MRI H} and of Theorem \ref{thm:existence}. For the uniqueness, from Theorem \ref{thm:existence}, it is sufficient to check that the operators $K_t^* \colon \M(\olom) \to H_t$ are injective for a.e.~$t \in [0,1]$. To this end, choose $t \in [0,1]$ such that $\supp \sigma_t$ has non empty interior, and let $\rho \in \M(\olom)$ be such that $K_t^* \rho = 0$. In particular $\F(c_j \rho) = 0$ in $\supp \sigma_t$, for every $j = 1, \dots, N$. Since $\F(c_j \rho)$ is analytic and $\supp \sigma_t$ contains an open ball, we conclude that $\F(c_j \rho) = 0$ in $\R^2$. By injectivity of the Fourier transform we have that $c_j \rho = 0$, and since we are assuming that $c(x) \neq 0$ for every $x \in \olom$, we conclude that $\rho =0$.
\end{proof}

\section{Conclusions and perspectives}

\label{sec:conclusions}

\medskip

In the paper, we have shown that it is possible to successfully use energy
functionals that are associated with a dynamic formulation of optimal
transport as regularization functionals for dynamic
inverse problems that aim at the recovery of measure-valued curves.
Let us point out some future directions of research. On the one hand,
the focus of the paper is on regularizers that penalize mass transport
by the squared distance as well as mass growth in terms of quadratic
costs for growth rate. Thus, a generalization to other convex optimal
transport energies (such as, e.g., the $p$-th power of the euclidean
distance) in an appropriate dynamic context (i.e., where the dynamic
formulation involves a continuity type equation), would be interesting
and seems to be possible. On the other hand, the regularized problems
involve, in addition to the transport energy, a Radon-norm term which
corresponds to a penalization of the total mass. Also here, a
generalization to other regularization functionals should be possible,
provided that one can still ensure boundedness of the total
mass. This way, it might be possible to impose, e.g., spatial
smoothness of the solution curve.

Finally, we would like to mention that a numerical optimization algorithm for the solution of the regularized problem is currently under preparation, in the general setting and also with focus on the application to dynamic MRI. The numerical solution to the dynamic optimal transport problem outside the inverse problems context is addressed for example in \cite{bb,chizat,ppo}, where the authors employ proximal splitting algorithms, after a careful discretization of the continuity equation constraint by means of staggered grids, and, subsequently, of the energy. The recent application to dynamic PET imaging \cite{schmitzer} also builds on these algorithms. The method we propose, based on Frank-Wolfe-type algorithms \cite{bredies-lorenz,duval-peyre,frankwolfe}, is inherently discretization-free and therefore genuinely cast in the space of Radon measures, in the spirit of \cite{bp,walter}. Generally, such %
algorithms are attractive for obtaining sparse solutions for inverse problems in terms of extremal points of the regularizer \cite{chambolle,bredies-carioni}.
In our setting,
the idea is to linearize the fidelity term in \eqref{eq:intro_tikh} around some initial guess, and then proceed to minimize (a suitably modified version) of the functional obtained. The key part of the analysis is that such linearized problem admits a solution which is an extremal point of the unit ball of the regularizer, making the minimization problem numerically accessible, although non-convex in general. Indeed, as shown in \cite{extremal}, the extremal points of the Benamou-Brenier regularizer are given by measures concentrated on curves (with a certain regularity), yielding an optimization problem in some Sobolev space. It is then possible to show convergence of the algorithm in the space of measures. Such an approach is then particularly well-suited to recover sparse solutions, given by travelling Dirac deltas with the addition of noise, making it attractive for medical applications in which the tracking of point-sources is relevant. %
With this approach one can, in principle, deal also with the unbalanced optimal transport case. The key ingredient is of course the characterization of the extremal points of the Wasserstein--Fisher--Rao energy, which is currently under preparation by the authors, and bases on a measure-theoretic superposition principle for the non-homogeneous continuity equation, which is in itself a novel result.

\section*{Acknowledgments}

The authors gratefully acknowledge support by the Christian Doppler
Research Association (CDG) and Austrian Science Fund (FWF) through the
Partnership in Research project PIR-27 ``Mathematical methods for
motion-aware medical imaging''. The Institute of Mathematics and
Scientific Computing, to which the authors are affiliated, is a member
of NAWI Graz (\url{http://www.nawigraz.at/en/}). The authors are further
members of/associated with BioTechMed Graz
(\url{https://biotechmedgraz.at/en/}).

\bibliography{bibliography}

\bibliographystyle{my_plain}

\appendix

\section{Measure theory} \label{app:measure}

\subsection{Measure theory preliminaries} \label{app:measure:prel}

In this paper we follow the definitions and notations of \cite{afp}. In particular, scalar or vectorial measures will always be defined on the Borel $\sigma$-algebra $\B(X)$ of some locally compact, separable metric space $X$. Given a measure $\mu$, we denote with $|\mu|$ its \textit{total variation}. We always assume that $|\mu|$ is at least locally finite. The set of $\R^m$-valued measures for which $|\mu|(X)<\infty$ is denoted by $\M(X;\R^m)$ and $\M(X):=\M(X;\R)$, while the set of positive finite measures is denoted by $\M^+(X)$.

\subsubsection{Absolute continuity, support and restriction}
Given $\mu \in \M(X)$ and $\nu \in \M(X;\R^m)$, we say that $\nu$ is \textit{absolutely continuous} with respect to $\mu$, in symbols $\nu \ll \mu$, if $|\nu|(E)=0$ whenever $\mu(E)=0$, $E \in \B(X)$. If $\mu$ and $\sigma$ are real or vector valued measures, we say that they are mutually singular, $\mu \perp \nu$, if there exists a set $E \in \B(X)$ such that $|\mu|(E)=0$ and $|\nu|(X \smallsetminus E)=0$.  For a measure $\mu$ its \textit{support}, denoted by $\supp \mu$, is the closure of the set of points $x \in X$ such that $|\mu|(U)>0$ for every neighbourhood $U$ of $x$. 
If $A \in \B(X)$ and $\mu \in \M(X;\R^m)$, the \textit{restriction} of $\mu$ to $A$ is the measure $\mu \zak A$ defined as $\mu \zak A (E):= \mu(A \cap E)$ for every $E \in \B(X)$.

\subsubsection{Push-forward}
Let $Y$ be a locally compact, separable metric space. A map $f \colon X \to Y$ is \textit{Borel} if  $f^{-1} (E) \in \B(X)$ for each $E \in \B(Y)$. If $\mu$ is a real or vector valued measure on $X$ we define the push-forward of $\mu$ through $f$ as the measure $f_\# \mu$ on $Y$, defined by $f_\# \mu(E):=\mu(f^{-1}(E))$ for each $E \in \B(Y)$. If $\f$ is a real or vector valued map on $Y$ integrable with respect to $f_\# \mu$ then
$
\int_Y \f \, d (f_\# \mu) = \int_X \f \circ f \, d\mu \,.
$
We recall that is $f$ is continuous and \textit{proper} ($f^{-1}(K)$ is compact if $K \subset Y$ is compact) and $\mu$ is finite, then also $f_\#\mu$ is finite.%

\subsubsection{Convergences}
Let $\{\mu_n\}$ be a sequence of measures on $X$. We say that $\mu_n$ \textit{narrowly converges} to $\mu$, in symbols $\mu_n \weak \mu$, if $\int_X \f  \, d\mu_n \to \int_X \f  \, d\mu$
for all $\f \in C_b(X)$, i.e., $\f$ continuous and bounded.
We say that $\mu_n$ \textit{weak* converges} to $\mu$, in symbols $\mu_n \weakstar \mu$, if $\int_X \f  \, d\mu_n \to \int_X \f  \, d\mu$ for all $\f \in C_0(X)$. We recall that $\f \in C_0(X)$ if $\f \in C(X)$ and for each $\e>0$ there exists $K \subset X$ compact such that $|\f|<\e$ in $X \smallsetminus K$.   Note that if $X$ is compact, then narrow convergence and weak* convergence coincide.

\subsubsection{Disintegration} \label{sec:disintegration}
Let $X$, $Y$ be locally compact, separable metric spaces. Consider $\{\mu_x \, \colon \, x \in X\}$ family of measures on $Y$. We say that the family $\{\mu_x\}$ is Borel if the map
$x \in X \mapsto \mu_x (B)$ is Borel measurable for every $B \subset Y$ measurable. Such condition implies that for every bounded Borel function $\f \colon X \times Y \to \R$, the map
$
x \in X \mapsto \int_Y \f (x,y) \, d\mu_x(y)
$	
is Borel measurable (see \cite[Prop 2.26]{afp}). We now state the disintegration theorem. For this version and the following properties, see \cite[Sections 2.3, 2.4]{abc} and \cite[Thm 2.28]{afp}.

\begin{theorem}[Disintegration] \label{thm:disint}
	Let $X$, $Y$ be locally compact, separable metric spaces. Let $\mu$ be a real (resp. vector valued) measure on $X \times Y$, $\pi \colon X \times Y \to X$ the projection on the first factor, and $\nu$ a positive measure on $X$, with the property that $\pi_\# |\mu| \ll \nu$. Then there exists a Borel family $\{\mu_x \, \colon \, x \in X\}$ of real (resp. vector valued) measures on $Y$ such that 
$\mu = \nu \otimes \mu_x$, that is,
\[
\int_{X \times Y} f(x,y) \, d\mu (x,y) = \int_X \left( \int_Y  f(x,y) \, d\mu_x(y) \right)  d\nu(x) \,,
\]
for every $f \in L^1(X \times Y; |\mu|)$.
A family $\{\mu_x\}$ such that $\mu=\nu \otimes \mu_x$ is called a disintegration of $\mu$ with respect to $\nu$. 
\end{theorem}

In the setting of Theorem \ref{thm:disint} the following properties hold:
\begin{enumerate}[(i)]
	\item If $\{\tilde{\mu}_x\}$ is another disintegration of $\mu$ with respect to $\nu$, then $\mu_x=\tilde{\mu}_{x}$ for $\nu$-a.e.~$x$,
\item If $\mu$ is finite, then also $\mu_x$ is finite for $\nu$-a.e.~$x$,
\item Let $E \in \B(X)$, $F \in \B(Y)$. Then $\mu(E \times F)=0$ if and only if $\mu_x(F)=0$ for $\nu$-a.e.~$x$. 
\end{enumerate}

\subsection{Narrow continuity results} \label{app:meas:narrow}

As in Section \ref{sec:continuity equation} we denote by $\curves$ the set of narrowly continuous curves $t \mapsto \rho_t$ and by $\pcurves$ the set of positive narrowly continuous curves. The remaining notations are the same as in Section \ref{sec:optimal}.

\begin{lemma} \label{lem:bound}
Let $\rho_t \in \pcurves$, $\rho:=dt \otimes \rho_t$, 
$m=dt \otimes (v_t \rho_t)$, $\mu = dt \otimes (g_t \rho_t)$, with $v_t \colon X \to \R^d$, $g_t \colon X \to \R$ measurable. Assume that $\de_t \rho + \Div m = \mu$
in the sense of \eqref{cont weak} and
\[
	E:= \int_X ( |v_t(x)|^2 + |g_t(x)|^2 ) \, d\rho_t(x) \, dt <\infty \,.
\]
Set $m :=\min_{t \in [0,1]} \rho_t (\overline{\Om}), M:=\max_{t \in [0,1]}\rho_t (\olom)$ and $C :=4 (m +E)$. Then $M \leq C$ and 
\[
\left|   \int_{\olom} \f \, (d\rho_t - d\rho_s) \right| \leq
\nor{\f}_{C^1} \sqrt{2CE} \, |t-s|^{1/2} \,,	 \quad \text{ for all } \,\, \f \in C^1(\overline{\Om}),\,\, 0 \leq s \leq t \leq 1\,.
\]
\end{lemma}

We remark that the above lemma is an easy generalization of Lemma 2.2 in \cite{kmv}, where the authors prove the same result under the restriction that $v_t=\nabla_x \, g_t$.

\begin{proof}
Since $\rho_t \in \pcurves$, we have $m,M<\infty$. Arguing as in the proof of Proposition~\ref{prop:cont rep}, we obtain that for all $\varphi \in C^1(\olom)$, the weak 
derivative of $\rho_t(\f):=\int_{\olom} \f \, d\rho_t$ satisfies
	\[
	\rho_t'(\f) = \int_{\olom} ( \nabla \f (x) \cdot v_t(x) + \f (x) g_t(x) ) \, d\rho_t (x) \,.
	\]
	In particular by applying twice the Cauchy--Schwarz inequality we get
\begin{align*}
	\left|  \rho_t(\f) - \rho_s (\f) \right| & = 
	\left| \int_s^t \int_{\olom} ( \nabla \f(x) \cdot v_\lambda (x) +  \f(x) g_\lambda (x)    ) \, d\rho_\lambda (x) \, d\lambda \right| \\ 
	& 
        \leq \sqrt{M}  \left(\int_s^t \int_{\olom} |\nabla \f(x) \cdot v_\lambda(x) + \f(x) g_\lambda(x)|^2 \, d\rho_\lambda(x) \, d\lambda \right)^{1/2} |t - s|^{1/2}  \\ 
	& \leq  \nor{\f}_{C^1} \sqrt{2M E} \, |t-s|^{1/2} \,. 
\end{align*}
The proof is concluded if we show that $M \leq C$. By applying the above estimate to $\f \equiv 1$ we get $|\rho_t(\overline{\Om}) - \rho_s(\overline{\Om}) | \leq \sqrt{2ME} $. If we pick $s$ and $t$ such that $\rho_s(\olom) = m$ and $\rho_t(\overline{\Om})=M$, then we get 
$M \leq m +  \sqrt{2ME} \leq m + \sqrt{2} (M+E)/2$, from which follows $M \leq 4 (m+E) = C$. 
\end{proof}

\begin{proposition} \label{prop:ub}
If $\rho_t \in \curves$ then  
$
\sup_{t \in [0,1]} \nor{\rho_t}_{\M(\overline{\Om})} < \infty \,.
$	
\end{proposition}

\begin{proof}
The curve $\rho_t$ defines a family $\{\rho_t\}_{t \in [0,1]}$ of functionals in $\M(\overline{\Om})$, via $\f \mapsto \int_{\olom} \f \, d\rho_t$. By narrow continuity, the map
	$t \mapsto \int_{\olom} \f \, d\rho_t$ is continuous for each $\f \in C(\overline{\Om})$, yielding that
	$\sup_{t \in [0,1]} \int_{\olom} \f \, d\rho_t < \infty$.
	The principle of uniform boundedness then implies the thesis.	
\end{proof}

\begin{proposition}[A refined version of Ascoli--Arzelà's Theorem] \label{ascoli}
  Let $K \subset \M(\overline{\Om})$ be sequentially weak* compact and $\rho^n_t \colon [0,1] \to \M(\overline{\Om})$ be such that $\rho_t^n \in K$ for all $t \in [0,1]$, $n \in \N$ and
\[	
\limsup_n   \|\rho_t^n - \rho_s^n\|_{C^1(\olom)^*} 
        \leq \omega (t,s) \, \text{ for every } \, t \in [0,1] \,,
\]
	where $\omega \colon [0,1] \times [0,1] \to [0,\infty)$ is continuous, symmetric and such that $\omega(t,t)=0$ for every $t \in [0,1]$. Then there exists a $C^1(\olom)^*$-continuous curve $\rho_t \colon [0,1] \mapsto \M (\overline{\Om})$ and a  subsequence $\{\rho_t^{n_k}\}$ such that
	$\rho_t^{n_k} \weakstar \rho_t$ for every $t \in [0,1]$.
\end{proposition}

The above statement is a particular case of \cite[Prop 3.3.1]{ags}, since $(\M(\overline{\Om}),\|\cdot\|_{C^1(\olom)^*})$ is a metric space and the $C^1(\olom)^*$-norm is weak* sequentially lower semicontinuous.

\smallskip
\noindent{\bfseries Proof of Proposition \ref{prop:cont rep}.}
Let $a \in C^\infty_c ((0,1))$, $b \in C^\infty (\overline{\Om})$. Set $\f(t,x):=a(t)b(x)$, so that $\f$ is a test function for \eqref{cont weak}. 
Define $\rho_t(b):=\int_{\olom} b(x) \, d\rho_t(x)$. Notice that $t \mapsto \rho_t(b)$ is measurable since $\rho_t$ is a Borel family. Moreover $t \mapsto \rho_t(b)$ belongs to $L^1((0,1))$ since 
$
\int_0^1 |\rho_t (b)| \, dt  \leq \nor{b}_{C^1(\olom)} \nor{\rho}_{\M(X)}$
where $\nor{b}_{C^1(\olom)} := \max \{\nor{b}_{\infty},\nor{\nabla b}_{\infty}\}$. Testing \eqref{cont weak} against $\f$ yields
that $\rho_t'(b)= \int_{\olom} \left[\nabla b(x) \cdot v_t(x) + b(x) g_t(x) \right]\, d\rho_t(x)$ weakly. 
Notice that for a.e.~$t \in [0,1]$ we have
\begin{equation} \label{var}
|\rho_t'(b)| \leq \nor{b}_{C^1 (\olom)} \, V(t)  \,, \quad V(t):= \int_{\olom} \left(|v_t(x)| + |g_t(x)| \right)\, d\rho_t(x) \,.
\end{equation}
In particular $V \in L^1((0,1))$ by assumption \eqref{cont int}, so that $\rho_t(b) \in W^{1,1}((0,1))$ with
\begin{equation} \label{sobolev}
\nor{\rho_t(b)}_{W^{1,1}((0,1))} \leq \nor{b}_{C^1 (\olom)} \left(  \nor{\rho}_{\M(X%
)} + \int_0^1 V(t) \, dt \right) \,.
\end{equation}
By the embedding $W^{1,1} ((0,1)) \hookrightarrow C([0,1])$, there exists a unique $\tilde{\rho}_t(b) \in C([0,1])$ such that $\tilde{\rho}_t(b)= \rho_t (b)$ for a.e.~$t \in [0,1]$, and
	$\nor{\tilde{\rho}_t(b)}_{\infty} \leq C \nor{b}_{C^1(\olom)}$ where $C >0$ does not depend on $b$, thanks to \eqref{sobolev}. Moreover
	\begin{equation} \label{differences}
	\tilde{\rho}_t(b) -\tilde{\rho}_s(b) =\int_{s}^t \rho'_\lambda(b) \, d\lambda	 \,\, \text{ for all } \,\, s,t \in [0,1] \,.
	\end{equation}
By density of $C^{\infty}(\olom)$ in $C^1(\olom)$, for each $t \in [0,1]$ the map $b \mapsto \tilde{\rho}_t(b)$ can be uniquely extended to an element of ${C^1(\olom)}^*$, since the maps $b \mapsto \rho_t(b)$ are linear and the extensions $\tilde{\rho}_t(b)$ are unique. This defines a bounded curve $t \mapsto \tilde{\rho}_t$ in ${C^1(\olom)}^*$. Such curve is uniformly continuous in $[0,1]$, since \eqref{var} and \eqref{differences} imply
$
|\tilde{\rho}_t(b) -\tilde{\rho}_s(b)| \leq \nor{b}_{C^1(\olom)} \, \int_s^t V(\lambda) \, d\lambda$ for every $b \in C^1(\olom)$.
We are left to prove that $\tilde{\rho}_t \in \M^+(\overline{\Om})$ for each $t \in [0,1]$. This follows from the fact that there is a Borel set $E \subset [0,1]$ with $|[0,1] \smallsetminus E|=0$ such that  $\rho_t \in \M^+(\overline{\Om})$ for every $t \in E$ and $\{\rho_t\}_{t \in E}$ is weak* sequentially precompact in $\M^+(\olom)$, since by Proposition \ref{prop:dis} we have that $\{\rho_t(\olom)\}$ is a.e.~bounded as $t \mapsto \rho_t(\olom)$ belongs to $BV((0,1))$. 
The weak* continuity of the curve $t \mapsto \tilde{\rho}_t$ in $\M^+(\overline{\Om})$ automatically follows from the one in ${C^1(\overline{\Om})}^*$: indeed let $\f \in C(\olom)$ and let $\{\f_n\}$ be a sequence in $C^1(\olom)$ such that $\nor{\f_n - \f}_{\infty} \to 0$ as $n \to \infty$. Then it is immediate to check that
$
\sup_{t \in [0,1]} | \tilde{\rho}_t(\f_n) - \tilde{\rho}_t (\f)  |\to 0$
so that $t \mapsto \tilde{\rho}_t (\f)$ is continuous, since it is uniform limit of continuous maps $t \mapsto \tilde{\rho}_t (\f_n)$. Finally let $\f \in C^1_c([0,1]\times \overline{\Om})$, $0 \leq t_1 \leq t_2 \leq 1$ and define $\f_\e (t,x):=a_\e(t) \f(t,x)$, where $a_\e \in C^\infty_c((t_1,t_2))$ is such that $0 \leq a_\e(t) \leq 1$, $\lim_{\e \to 0} a_\e (t) = \rchi_{(t_1,t_2)}(t)$ for almost every $t \in [0,1]$ and $\lim_{\e \to 0} a'_\e = \delta_{t_1} - \delta_{t_2}$ weakly* in $\M([0,1])$.
Testing \eqref{cont weak} against $\f_\e$ and passing to the limit as $\e \to 0$ (by continuity of $t \mapsto \int_{\olom}\f(t,x) \, d\tilde{\rho}_t (x)$ and \eqref{cont int}) yields \eqref{cont boundary}.

\section{Time-Dependent Bochner Spaces} \label{app:bochner}

In this appendix we assume \ref{H1}--\ref{H3} as in Section \ref{sec:func_sett}. Definitions of step functions, strong measurability, weak measurability, separably valued and integrability are as in Sections \ref{sec:meas}, \ref{sec:int}.

\subsection{Auxiliary results and proofs of Section~\ref{sec:hilbert setting}} \label{app:bochner:auxiliary}

Here we state and prove a suitable version of Egoroff's Theorem, as well as present the proofs of Theorems \ref{thm:pettis}, \ref{thm:dominated_convergence}, \ref{thm:completeness}.

\begin{proposition}[Egoroff] \label{prop:egoroff}
	Let $f_n ,f \colon [0,1] \to H$ be strongly measurable and such that, for a.e.~$t \in [0,1]$, 
$\lim_n \norh{f_n (t) - f(t)} = 0$. Then for each fixed $\e >0$ there exists a Lebesgue measurable set $E \subset [0,1]$ with $|E|< \e$ and such that $f_n \to f$ uniformly in $[0,1] \smallsetminus E$, that is, 
	\begin{equation} \label{claim egoroff}
	\lim_n   \sup_{ t \in [0,1] \smallsetminus E} \norh{f_n(t) - f(t)} = 0 \,.
	\end{equation}
\end{proposition}

\begin{proof}
The proof follows by replacing absolute values with the $H_t$ norms in the proof of the classic Egoroff Theorem. Indeed, since $f_n, f \colon [0,1] \to H$ are assumed to be strongly measurable, the map $t \mapsto \norh{f_n(t)-f(t)}$ is Lebesgue measurable (see Remark \ref{rem:strong}). Then the sets $E_k (n) := \cup_{m \geq k} \{ t \in [0,1] \, \colon \, \norh{f_m(t)-f(t)} \geq 1/n \}$ are measurable for each fixed $n,k \in \N$. Moreover, for $n$ fixed, we have that $E_{k+1} (n) \subset E_k(n)$ and $|E_k(n)| \searrow 0$ as $k \to \infty$, since we are assuming that $f_n \to f$ a.e.~in $[0,1]$. Let $\{k_n\}$ be an increasing sequence of indices such that $|E_{k_n} (n)|<\e / 2^n$. It is immediate to see that the measurable set $E:= \cup_{n=1}^{\infty} E_{k_n} (n)$ satisfies \eqref{claim egoroff}.
\end{proof}

\noindent{\bfseries Proof of Theorem \ref{thm:pettis}.} \textit{Part 1.}
Assume that $f$ is strongly measurable. Hence there exists a sequence $\{f_n\}$ of step functions $f_n \colon [0,1] \to D$ with $f_n= \sum_{j=1}^{N_n} \rchi_{E_{n,j}} \f_{n,j}$ such that 
$\norh{ i_t f_n (t) - f(t) } \to 0$ a.e. in $[0,1]$.
We claim that $f$ is weakly measurable: Indeed fix $\f \in D$ and define $\theta(t):= \inner{i_t \f}{f(t)}$, $\theta_n (t):=\inner{i_t \f}{i_t f_n(t)}$ for $t \in [0,1]$. Clearly $\theta_n$ is measurable for fixed $n$, by \ref{H3}. Moreover using Cauchy--Schwarz and \ref{H1} yields
$
| \theta_n (t) - \theta(t)|  \leq 
C \nor{\f}_D  \norh{i_t f_n (t) -f(t)}$,
so that $\theta_n \to \theta$ for a.e.~$t \in [0,1]$, implying that $\theta$ is measurable, and hence $f$ is weakly measurable. We will now show that $f$ is essentially separably valued. By definition $i_t f_n$ is strongly measurable and $i_tf_n \to f$ a.e., hence Proposition \ref{prop:egoroff} implies that for every $n \in \N$ there exists a measurable set $E_n \subset [0,1]$ with 
$|E_n|\leq 1/n$ and such that 
$\norh{ i_t f_n (t) - f(t) } \to 0$ uniformly on $[0,1] \smallsetminus E_n$.
Define the countable set 
$S := \cup_{n=1}^\infty f_n ([0,1])  \subset D$.
Let $E:= \cap_{n=1}^\infty E_n$ so that $|E|=0$. Fix $\e>0$ and $t \in [0,1] \smallsetminus E$. Hence there exists an index $n \in \N$ such that $t \in [0,1] \smallsetminus E_n$. By uniform convergence we conclude that 
$\norh{ i_t f_n (t) - f(t) } < \e$, for sufficiently large $n$. Therefore Definition \ref{def:strong} iii)~is satisfied by setting $\f := f_n (t)$.

\smallskip
\noindent \textit{Part 2.} Let $f$ be weakly measurable and essentially separably valued. Let $S =\{ \f_n \}\subset D$ be countable and $E \subset [0,1]$ with $|E|=0$ satisfying Definition \ref{def:strong}.
For $t \in [0,1]$ define
$\psi_n(t) :=1/\norh{i_t \f_n}$ if $i_t \f_n \neq 0$ and $\psi_n(t):=0$ otherwise.
Notice that $\psi_n$ is Lebesgue measurable for each fixed $n$, since $t \mapsto \norh{i_t \f_n}$ is measurable by \ref{H3}. We will now show that
\begin{equation} \label{dim6}
	\norh{f(t)} = \sup_n \psi_n(t) |\inner{i_t \f_n}{f(t)}    |  \quad \text{ for all} \quad t \in [0,1] \smallsetminus E \,.
\end{equation}
Indeed, the supremum never exceeds $\norh{f(t)}$ by the Cauchy--Schwarz inequality. Conversely, if $f(t)=0$ the equality is trivial, hence assume $f(t) \neq 0$. Fix $0<\e < \norh{f(t)}/2$. Since $f$ is essentially separably valued, there exists $\f_n \in S$ such that $\norh{i_t \f_n  - f(t)}<\e$. In particular  $i_t \f_n \neq 0$. Then
\[
\begin{aligned}
\norh{f(t)} & < \e +  \norh{i_t \f_n} = \e + \psi_n (t) | \inner{i_t \f_n}{i_t \f_n} | \\
  & \leq \e +  
   \psi_n (t) | \inner{i_t \f_n}{i_t\f_n - f(t)} | + \psi_n (t) | \inner{i_t \f_n}{f(t)} |\\
   & \leq  2\e + \psi_n (t) | \inner{i_t \f_n}{f(t)} |  \leq 2\e + \sup_n \psi_n (t) | \inner{i_t \f_n}{f(t)} | \,,
\end{aligned}
\]
and since $\e$ is arbitrarily small we conclude. Notice that the map $t \mapsto |\inner{i_t \f_n}{f(t)} |$ is measurable by weak measurability of $f$. Thus 
$
 t \mapsto  \psi_n (t) |\inner{i_t \f_n}{f(t)}    |  
$
is measurable, being product of measurable maps. Since the countable pointwise supremum of measurable functions is measurable, by \eqref{dim6} we conclude that $t \mapsto \norh{f(t)}$ is measurable. Also the map $\theta_n (t):= \norh{f(t) - i_t \f_n}$ is measurable at $n$ fixed, as
\[
\norh{ f(t) - i_t \f_n}^2 = \norh{f(t)}^2 - 2 \inner{i_t \f_n}{f(t)} + \norh{i_t \f_n}^2
\]
is a sum of measurable functions, where the second element is measurable by weak measurability of $f$ and the third by \ref{H3}. Fix $\e >0$ and define the measurable sets
$E_n:=\{ t \in [0,1] \, \colon \, \theta_n(t) < \e \}$ and the map
$g \colon [0,1] \to D$ by setting
$g(t) := \f_n$  if $t \in E_n \smallsetminus \cup_{j=1}^{n-1} E_j$ for some $n$, and $g(t)=0$ otherwise. 
Note that for $t \in [0,1] \setminus E$ there exists some index $n$ such that $\norh{f(t)-i_t \f_n}<\e$. Therefore, by picking the smallest $n$ such that this condition is verified, we have $g(t)=\f_n$. Since this is true for each $t \in [0,1] \setminus E$, this means that $\norh{f(t)-i_t g(t)}<\e$ a.e.~in $[0,1]$.
Hence we can approximate $f$ essentially uniformly by $i_t g(t)$ with $g$ countably valued. By choosing $\e = 1/n$ we obtain countably valued functions $g_n \colon [0,1] \to D$ such that 
\begin{equation} \label{dim11}
\norh{f(t) - i_t g_n(t)}<1/n \quad \text{ for every } \,\, t \in E \,,
\end{equation}
where $|[0,1] \setminus E|=0$. Note that by definition $g_n = \sum_{j=1}^\infty \rchi_{E_{n,j}} \f_{n,j}$ with $\{E_{n,j}\}_{j \in \N}$ measurable partition of $[0,1]$. Therefore for every $n \in \N$, there exists $k_n$ such that the set $\cup_{j=1}^{k_n} E_{n,j}$ satisfies 
\begin{equation} \label{sets Fn}
\left|[0,1] \smallsetminus \cup_{j=1}^{k_n} E_{n,j} \right|< \frac{1}{n^2} \,.
\end{equation}
Set 
$
F_n:=\cap_{s=n}^{\infty} \cup_{j=1}^{k_s} E_{s,j}$, $F:=\cup_{n=1}^\infty F_n$.
In this way $|[0,1] \smallsetminus F|=0$ by \eqref{sets Fn}, since
$
\left| [0,1] \smallsetminus F_n \right| \leq \sum_{s=n}^{\infty} \frac{1}{s^2} \to 0$ as $n \to \infty$. 
Now define step functions $f_n \colon [0,1] \to D$ obtained by truncating $g_n$, that is, $f_n :=\sum_{j=1}^{k_n} \rchi_{E_{n,j}} \f_{n,j}$. If we prove that 
\begin{equation} \label{claim truncation}
\lim_n \norh{f(t) - i_t f_n(t)} = 0 \,\, \text{ for every } \,\, t \in F \cap E\,,
\end{equation}
then we conclude that $f$ is strongly measurable, since $|[0,1]\smallsetminus (F \cap E)|=0$. %
In order to show \eqref{claim truncation}, fix $\e>0$ and $t \in F \cap E$. By using \eqref{dim11} and \ref{H1} we have that for all $n \in \N$
\begin{equation} \label{claim  trunc}
\begin{aligned}
\norh{f(t) - i_t f_n(t)} & \leq \norh{f(t) - i_t g_n(t)} + 
\norh{i_t g_n(t) - i_t f_n(t)} \\ & < \frac{1}{n} + C \nor{\sum_{j=k_n +1}^{\infty} \rchi_{E_{n,j}} (t) \, \f_{n,j}}_{D} \,.
\end{aligned}
\end{equation}
Since $t \in F$, by definition there exists an index $N_t$ such that $t \in F_{N_t}$. Hence $t \in \cup_{j=1}^{k_n} E_{n,j}$ for every $n \geq N_t$, so that 
$\sum_{j=k_n +1}^{\infty} \rchi_{E_{n,j}} (t) \, \f_{n,j} = 0$ for each $n \geq N_t$.
Set $n_{\e,t} :=\max\{ N_t, 1/\e\}$. From \eqref{claim  trunc} we have
$\norh{f(t) - i_t f_n(t)} < \e$ for every $n \geq n_{\e,t}$, implying
\eqref{claim truncation}.

\medskip

\noindent{\bfseries Proof of Theorem \ref{thm:dominated_convergence}.} 
Since $f_n \to f$ a.e., the map $f$ in strongly measurable and $\theta_n(t):= \norh{f_n(t) - f(t)}$ is Lebesgue measurable. By assumption we have that $\theta_n \to 0$ and $\theta_n \leq 2 g$ a.e.~in $[0,1]$. Therefore, by the classic dominated convergence theorem, each $\theta_n$ is integrable and $f_n \to f$ strongly in $L^1([0,1];H)$. %
To conclude integrability for $f$ it is sufficient to employ triangle inequality, integrability of $\theta_n$ and Theorem \ref{thm:int}. %

\medskip

\noindent {\bfseries Proof of Theorem \ref{thm:completeness}.}
The fact that $\nor{\cdot}_{L^p}$ is a norm follows immediately from the classic case as well as the fact that the map $t \mapsto \norh{f(t)}^p$ is measurable for each $p \geq 1$, when $f$ is assumed to be strongly measurable (see Remark \ref{rem:strong}). Moreover $\innerL{\cdot}{\cdot}$ is an inner product, since the spaces $H_t$ are Hilbert. In order to show completeness, it is sufficient to follow the lines of the proof of the classic Riesz--Fischer theorem. 
Let $1 \leq p < \infty$
and let $\{f_n\}$ be a Cauchy sequence in $L^p([0,1];H)$. In any normed linear space, a Cauchy sequence having a convergent subsequence converges to the same limit. Therefore, up to extracting a subsequence, we can assume that
\begin{equation} \label{cauchy1}
\nor{f_m-f_n}_{L^p} < \frac{1}{2^n} \,\, \text{ for every } \,\, n \,\, \text{ and } \,\, m >n \,.
\end{equation}
For every $n \in \N$ define measurable sets $E_n:=\{ t \in [0,1] \, \colon \, \norh{f_{n+1}(t)-f_n(t)} \geq 1/n^2 \}$, so that $\rchi_{E_n}(t)/n^2 \leq \norh{f_{n+1}(t)-f_n(t)}$ a.e.~in $[0,1]$ and $|E_n|< n^{2p}/2^{np}$ by \eqref{cauchy1}. In particular one has $\sum_n |E_n| < \sum_n n^{2p}/2^{np} < \infty$. Define $F_n:=\cup_{m \geq n} E_m$, so that $\{F_n\}$ is a nested sequence of measurable sets, with $|F_n| \leq \sum_{m \geq n}m^{2p}/2^{mp} \to 0$ as $n \to \infty$. Finally set $F:=\cap_n F_n$, which satisfies $|F|=0$. By definition, if $t \in [0,1] \smallsetminus F$, then 
$\norh{f_{n+1}(t)-f_n(t)} < n^{-2}$ for $n$ sufficiently large.
Hence for $t \in [0,1] \smallsetminus F$, $m >n$ and $n$ sufficiently large one has
$\norh{f_{m}(t)-f_n(t)} \leq  \sum_{j \geq n} \norh{f_{j+1}(t)-f_j(t)} < \sum_{j \geq n} 1/j^2$,
and since $\sum_{j \geq n} 1/j^2 \to 0$ as $n \to \infty$, we conclude that $\{f_n(t)\}$ is a Cauchy sequence in $H_t$. For $t \in  [0,1] \smallsetminus F$ denote by $f(t) \in H_t$ the strong  limit of $\{f_n(t)\}$, which exists since $H_t$ is complete. For $t \in F$ set $f(t)=0$. This defines a map $f \colon [0,1] \to H$, which is strongly measurable since it is the a.e.~pointwise limit of a sequence of strongly measurable maps (see Remark \ref{rem:strong}). Moreover, by the a.e.~pointwise convergence, we also have that $\norh{f_n(t)} \to \norh{f(t)}$ as $n \to \infty$ for a.e.~$t$. Since the maps $t \mapsto \norh{f_n(t)}, t \mapsto \norh{f(t)}$ are measurable, we can apply Fatou's Lemma and conclude that%
$
\int_{0}^1 \norh{f(t)}^p \, dt \leq \liminf_n \int_{0}^1 \norh{f_n(t)}^p \, dt$, which is bounded since $\{f_n\}$ is a Cauchy sequence in $L^p([0,1];H)$. Hence $f \in L^p([0,1];H)$. Finally, one more application of Fatou's Lemma combined with %
\eqref{cauchy1} yields $\nor{f_n-f}_{L^p} \to 0$ as $n \to \infty$.

\subsection{Comparison with classic Bochner theory} \label{app:bochner:comparison}

In this section we will investigate integrability properties for $i_t^*f \colon [0,1] \to D^*$ when 
$f \in L^1([0,1];H)$. Since the codomain of $i_t^* f$ is the fixed space $D^*$, it makes sense to check whether $i_t^* f$ is integrable in a classic sense. Specifically, in Proposition \ref{prop:gelfand}, we will see that 
$i_t^* f$ is always Gelfand integrable.
On the other hand, $i_t^*f$ is not always Bochner integrable, as we show in Example \ref{ex2}. The main impediment is that $i_t^*f$ is not strongly measurable in general. Finally in Proposition \ref{prop:bochner} we will give sufficient conditions under which $i_t^*f$ is Bochner integrable.

\begin{proposition}\label{prop:gelfand}
Assume that $f \in L^1([0,1];H)$. Then $i_t^*f$ is Gelfand integrable in $D^*$, that is, for each measurable set $E \subset [0,1]$ there exists an element $I_E(i_t^*f) \in D^*$ such that
\begin{equation} \label{show1}
\inn{I_E(i_t^*f)}{\f}_{D^*,D} = \int_E \inn{i_t^* f}{\f}_{D^*,D} \,dt \,, \,\,\, \text{ for all }\,\, \f \in D\,. 
\end{equation}
\end{proposition}

\begin{proof}
Let $\f \in D$. By duality one has
$\inn{i_t^*f(t)}{\f}_{D^*,D} = \inner{f(t)}{i_t \f}$.
Therefore $i_t^* f$ is weak* measurable since $f \colon [0,1] \to H$ is weak measurable. By \ref{H1},
	\[
	\int_0^1 |\inn{i_t^* f(t)}{\f}_{D^*,D}| \,dt = 
	 \int_0^1 |\inner{i_t\f}{f(t)}| \,dt
	\leq C \nor{\f}_D  \int_0^1 \norh{f(t)} \, dt < \infty
	\]
	since $f$ is integrable. This shows that $t \mapsto \inn{i_t^* f(t)}{\f}_{D^*,D}$ belongs to $L^1([0,1])$ for each $\f \in D$. Hence $i_t^*f$ is Gelfand integrable by Theorem 11.52 in \cite{aliprantis}, and \eqref{show1} holds.
\end{proof}

\begin{example}[Radial sampling] \label{ex2}
Let $\Om:=B_1(0)=\{ x\in \R^2 \, \colon \, |x|<1\}$ and for $t \in [0,1]$ define the lines through the origin
$S_t :=\{ ( \cos(\pi t)s, \sin (\pi t) s) \, \colon \,  |s| <1 \}$, so that $S_t \subset \Om$. Define $D=C_0(\Om)$ equipped with the supremum norm. Hence $D^*=\M(\Om)$. Define $H_t:=L^2_{\sigma_t}(S_t)$ with inner product $\inn{h_1}{h_2}_{H_t}:=\int_{S_t} h_1  h_2 \, d\sigma_t$, $\sigma_t:=\mathcal{H}^1 \zak S_t$ where $\mathcal{H}^1$ is the 1-dimensional Hausdorff measure. Finally define $i_t \colon D \to H_t$ by $i_t \f = \f|_{S_t}$. 
It is straightforward to check that \ref{H1}-\ref{H3} are satisfied, so that we can consider the space $\Ltwo$ defined as in \eqref{L2}.

We will now construct a map $f$ belonging to $\Ltwo$, compute the Gelfand integral of $i_t^*f \colon [0,1] \to D^*$ and show that $i_t^*f$ is not Bochner integrable.  To this end, notice that for a map $h \colon \Om \to \R$ such that $h/|x| \in L^1(\Om)$ we have that
\begin{equation} \label{coarea radial}
\int_0^1 \int_{S_t} h \, d\mathcal{H}^1  dt =
\int_\Om \frac{h}{\pi |x|} \, dx \,.
\end{equation}
Note that \eqref{coarea radial} is an easy consequence of the classical coarea formula \cite[Thm 3.11]{evansgariepy}, and its proof is left to the reader.
Let now $\tilde{f} \colon \Om \to \R$ be such that $\tilde{f} \not\equiv 0$ and $\tilde{f}/|x| \in L^2(\Om)$. 
By applying \eqref{coarea radial} to $|\tilde{f}|^2$ and by the assumptions on $\tilde{f}$ we have that $\tilde{f}|_{S_t}$ belongs to $H_t$ for a.e.~$t \in [0,1]$. 
Define $f \colon [0,1] \to H$ by 
$f(t):=\tilde{f}|_{S_t}$. 
Notice that $f$ is strongly measurable, since $\tilde{f}$ can be approximated in $\Om$ by $C_0$ functions. Moreover by \eqref{coarea radial} we infer
\[
\int_0^1 \norh{f(t)}^2 \, dt =
\int_0^1 \int_{S_t} |\tilde{f}|^2\, d\mathcal{H}^1 dt =
\int_\Om \frac{|\tilde{f}|^2}{\pi |x|} \, dx \,,
\]
which is finite by assumption on $\tilde{f}$. Hence $f$ is integrable by Theorem \ref{thm:int}, and it belongs to $\Ltwo$. The Gelfand integral of $i_t^*f$, which exists by Proposition \ref{prop:gelfand}, is given by 
\[
I(i_t^* f)=\frac{\tilde{f}}{\pi |x|} \, d\mathcal{H}^2 \zak \Om\,.
\]
The above follows immediately by applying \eqref{coarea radial} to $\f \tilde{f}$ with $\f \in D$.
However $i_t^*f$ is not Bochner integrable, since it is not strongly measurable in the classic sense \cite[Ch II]{diestel}: For every $E \subset [0,1]$ with $|E|=0$ we have that the set $i_t^*f ([0,1] \smallsetminus E)$
is not norm separable in $\M(\Om)$. Indeed, it is easy to show that for a.e.~$w \neq t$ 
\[
\nor{f(t) \mathcal{H}^{1} \zak S_t - f(w) \mathcal{H}^{1} \zak S_w}_{\mathcal{M}(\Om)} = \int_{S_t} |\tilde{f}| \, d \mathcal{H}^{1} + \int_{S_w} |\tilde{f}| \, d \mathcal{H}^{1} \,.  
\]
Since $\tilde{f} \not\equiv 0$, we infer that $i_t^*f (F)$ is a discrete set for any $F \subset [0,1]$ with $|F|>0$. Therefore $i_t^*f (F)$ is norm separable if and only if $F$ is countable, which is never the case. 
Hence $i_t^*f$ is not essentially separably valued, and the classic Pettis Theorem \cite[Ch II.1, Thm 2]{diestel} implies that $i_t^*f$ is not strongly measurable and hence not Bochner integrable.  
  \end{example}

\begin{proposition} \label{prop:bochner}
Assume that $i_t^* i_t (S) \subset D^*$	is essentially norm separable for each countable set $S \subset D$ and that $D$ is reflexive.
If $f \in L^1([0,1];H)$ then $i_t^* f \colon [0,1] \to D^*$ is Bochner integrable.
\end{proposition}

\begin{proof}
Since $f \colon [0,1] \to H$ is strongly measurable, it is also weakly measurable and essentially separably valued by Theorem \ref{thm:pettis}. 
We start by showing that weak measurability for $f$ implies weak measurability for $i_t^*f$ in the classic sense. Indeed by reflexivity the canonical injection $j \colon D \to D^{**}$ is also surjective. Therefore for each $\f^{**} \in D^{**}$ there exists a unique $\f \in D$ with $j(\f)=\f^{**}$ and we have 
$\inn{\f^{**}}{i^*_t f(t)}_{D^{**},D^{*}}= \inner{f(t)}{i_t \f}$, which is measurable since $f$ is weakly measurable. Now let 
$S=\{\f_n\} \subset D$ and $E \subset [0,1]$ measurable with $|E|=0$ and such that Definition \ref{def:strong} is satisfied. Therefore for every $\e>0$ and $t \in [0,1] \smallsetminus E$ there exists $\f_n \in S$ such that $\norh{i_t \f_n -  f(t)} <  \e /C$. Therefore by \ref{H1} we can estimate
$\nor{i_t^*i_t \f_n - i_t^* f(t)}_{D^*} \leq 
C \, \norh{i_t \f_n -  f(t)} <  \e$.
Hence the points of $i_t^*f([0,1]\smallsetminus E )$ are arbitrarily close to 
$\{i_t^*i_t \f_n \,,\, n \in \N ,\, t \in [0,1] \smallsetminus E\}$.
Since $i_t^*i_t (S)$ is assumed to be essentially separable in $D^*$, there exists a measurable set $F \subset [0,1]$ with $|F|=0$ such that $\{i_t^*i_t \f_n \,,\, n \in \N ,\, t \in [0,1] \smallsetminus F\}$
is separable in $D^*$. By defining $\tilde{E}:=E \cup F$ we obtain that also $i_t^*f([0,1]\smallsetminus \tilde{E} )$ is norm separable, hence $i_t^*f$ is essentially separably valued in the classic sense. By the classic Pettis Theorem \cite[Ch II.1, Thm 2]{diestel} we conclude that $i_t^*f$ is strongly measurable. Finally (H2) implies that $\int_0^1 \nor{i_t^* f(t)}_{D^*} \, dt \leq C \int_0^1 \norh{f(t)} \, dt < \infty$. By \cite[Ch II.2, Thm 2]{diestel}, it follows that $i_t^*f$ is Bochner integrable.
\end{proof}

\end{document}